\documentclass[11pt]{article}
\usepackage{amsmath,amsfonts,amssymb,amsthm,amscd,color}

\title{The resolvent cocycle in twisted cyclic cohomology and a local index formula for the Podle\'s sphere}

\author{Adam Rennie, Roger Senior
\thanks{email: \texttt{adam.rennie@anu.edu.au},\ \texttt{roger.senior@anu.edu.au}}
 \\[6pt]
Mathematical Sciences Institute,
Australian National University\\
Acton, ACT, 0200, Australia\\[6pt]
}

\advance\topmargin by -\headheight
\advance\topmargin by -\headsep
\evensidemargin=\oddsidemargin
\makeatletter
\def\section{\@startsection{section}{1}{\z@}{-3.5ex plus -1ex minus
  -.2ex}{2.3ex plus .2ex}{\large\bf}}
\def\subsection{\@startsection{subsection}{2}{\z@}{-3.25ex plus -1ex
  minus -.2ex}{1.5ex plus .2ex}{\normalsize\bf}}
\makeatother

\numberwithin{equation}{section} 

\theoremstyle{plain} 
\newtheorem{thm}{Theorem}[section]
\newtheorem{lemma}[thm]{Lemma}
\newtheorem{prop}[thm]{Proposition}
\newtheorem{corl}[thm]{Corollary}

\theoremstyle{definition} 
\newtheorem{defn}[thm]{Definition}

\DeclareMathOperator{\res}{Res} 
\DeclareMathOperator{\Tr}{Tr}     

\newcommand{\s}{\sigma}      

\newcommand{\A}{\mathcal{A}}  
\newcommand{\B}{\mathcal{B}}  
\newcommand{\C}{\mathbb{C}}   
\newcommand{\D}{\mathcal{D}}  
\renewcommand{\H}{\mathcal{H}}  
\newcommand{\K}{\mathcal{K}}   
\newcommand{\cm}{\mathcal{M}} 
\newcommand{\cn}{\mathcal{N}} 
\newcommand{\N}{\mathbb{N}}   
\newcommand{\R}{\mathbb{R}}   
\newcommand{\T}{\mathbb{T}}  
\newcommand{\Z}{\mathbb{Z}}   

\newcommand{\ox}{\otimes}     

\newcommand{\bma}{\left(\begin{array}{cc}}
\newcommand{\ema}{\end{array}\right)}
\newcommand{\bca}{\left(\begin{array}{c}}
\newcommand{\eca}{\end{array}\right)}

\newcommand{\la}{\langle}
\newcommand{\ra}{\rangle}

\newcommand{\p}{\partial}
\newcommand{\cM}{{\mathcal M}}

\newcommand{\half}{\tfrac{1}{2}} 

\newcommand{\hideqed}{\renewcommand{\qed}{}} 

\newcommand{\enveq}[1]{\begin{equation}#1\end{equation}}
\newcommand{\enveqn}[1]{\begin{equation*}#1\end{equation*}}
\newcommand{\enval}[1]{\begin{align}#1\end{align}}
\newcommand{\envaln}[1]{\begin{align*}#1\end{align*}}
\newcommand{\cA}{\mathcal{A}}
\newcommand{\cB}{\mathcal{B}}
\newcommand{\cD}{\mathcal{D}}
\newcommand{\cH}{\mathcal{H}}
\newcommand{\cK}{\mathcal{K}}
\newcommand{\cN}{\mathcal{N}}
\newcommand{\cO}{\mathcal{O}}
\newcommand{\cP}{\mathcal{P}}
\newcommand{\cU}{\mathcal{U}}
\newcommand{\bR}{\mathbb{R}}
\newcommand{\bC}{\mathbb{C}}
\newcommand{\bZ}{\mathbb{Z}}
\newcommand{\bN}{\mathbb{N}}
\newcommand{\bT}{\mathbb{T}}

\newcommand{\re}[3]{t_{#2, #3}^{#1}}

\DeclareMathOperator{\Ind}{\mathrm{Ind}} 

\newcommand{\wop}{{\rm w{\text -}OP}}

\begin{document}

\maketitle

\vspace{-8pt}

\begin{abstract}
%
%
%
%

We continue the investigation of twisted homology theories in the context of dimension drop phenomena.
This work unifies previous equivariant index calculations in twisted cyclic cohomology. We do this by proving the existence of the resolvent cocycle, 
a finitely summable analogue of the JLO cocycle, under
weaker smoothness hypotheses and in the more general setting of `modular' spectral triples.
As an application we show
that using our twisted resolvent cocycle, we can obtain a local index formula for the Podle\'s sphere. The resulting
twisted cyclic cocycle has non-vanishing Hochschild class which is in dimension 2.

\end{abstract}


\parskip=0.08in

\section{Introduction}
\label{sec:intro}
This paper proves a residue index formula in noncommutative geometry for  `modular spectral triples', which are analogues of spectral 
triples with twisted traces. This is the appropriate setting for examples
arising from $q$-deformations which typically experience `dimension drop' in homology, \cite{H,HK,NT,SW,W}. The main results are as follows.

1) We show that for finitely summable modular spectral triples the resolvent cocycle exists, is continuous
and is an index cocycle under weaker smoothness conditions than have previously been used. 
In particular we do not need the pseudodifferential
calculus to establish these facts, so that we are free to replace the usual pseudodifferential calculus by other schemes later, in order to obtain 
local index formulae.

2) We show that  modular spectral triples have a well-defined pairing with equivariant $K$-theory. In
the finitely summable and weakly smooth case we show that this pairing can be computed using the resolvent cocycle,
which defines a twisted cyclic cocycle.

3) We apply the results of 1) and 2) to prove a local index formula for the Podle\'s sphere
in twisted cyclic cohomology. This index formula puts the results of several authors into a common framework, \cite{H,KW,W}.
In particular, the twisted Hochschild class of our residue cocycle is an explicit constant multiple of the fundamental Hochschild cocycle
for the Podle\'s sphere, \cite{H,KW}, and our explicit index pairings can be compared to those in \cite{W}.

The computations in 3) are similar to what was done in \cite{NT}, 
however they used the twisting by the modular
automorphism, rather than the inverse of the modular automorphism. While the summability
is the same in both cases, the twisted Hochschild homology for the modular automorphism is trivial
in dimension 2, while the inverse of the modular automorphism avoids the dimension drop. 
Thus the cocycle obtained in \cite{NT} is cohomologous to a $0$-cocycle, while ours is not. We 
also note
that in \cite{NT} the starting point was the JLO cocycle in entire cyclic cohomology rather than 
the resolvent cocycle.

The exposition is as follows.
In Section \ref{sec:KMS-index} we introduce the basic definitions for modular spectral triples, including smoothness and summability. 
We then show that a modular spectral triple defines an equivariant $KK$-class, and so gives us a well-posed $K$-theory valued index problem.
The remainder of Section 2 demonstrates
that together with a representative of an equivariant $K$-theory class, we obtain a well-posed numerical index problem. The aim of Section
\ref{sec:local-index} is then to show that these notions are compatible.


We address the existence, continuity and index properties of the resolvent cocycle in Section \ref{sec:local-index}.
We begin by looking at our weak smoothness condition, and
proving some basic results that follow from this assumption. Then we prove the existence and continuity of the resolvent cocycle, which
originated in \cite{CPRS2}, and show that it computes the numerical index. Finally we show, using results from \cite{KNR}, that this numerical index
is compatible with the $K$-theory valued index in a precise way.

In section \ref{sec:pods-chern} we show that the spectral triple introduced by \cite{DS} defines a 2-dimensional modular spectral triple, 
which is weakly smooth in our sense. Numerous results of \cite{KW,NT,SW,W} are incorporated into this statement.
We employ Neshveyev and Tuset's modification of the 
pseudodifferential calculus to obtain a version of the local index formula for the
Podle\'s sphere. Thus we see that with a suitable pseudodifferential calculus, our resolvent index formula can be extended to a 
full local index formula as in \cite{CPRS2,CM,Hig}. We conclude by computing some explicit index pairings, and as a corollary see that the 
degree two term in the residue index cocycle is not a coboundary.

{\bf Acknowledgements.} It is a pleasure to acknowledge the assistance of our 
colleagues Alan Carey, Ulrich Kr\"ahmer and Joe V\'{a}rilly. Both authors were supported by
the Australian Research Council.

\section{Modular spectral triples and equivariant $K$-theory}
\label{sec:KMS-index}

We begin this section by defining modular spectral triples, a generalisation of semifinite spectral triples, \cite{BeF,CP2,CPRS2}, allowing for 
twisted traces (weights) in place of traces. We then consider the index pairings defined by modular spectral triples.

The strategy to study index pairings is almost the same as in \cite{CPRS2,CPRS3}. Given a representative of an equivariant $K$-theory class for
an algebra $\A$, we show that a modular spectral triple over $\A$ allows us to formulate a well-defined 
(semifinite) index problem. By following the strategy of \cite{CPRS2,CPRS3}, we find that the index can be computed  
by pairing a cocycle with the Chern character of the $K$-theory class. 
\subsection{Modular spectral triples}\label{subsec:mod-specs}

Let $\cn$ be a semifinite von Neumann algebra acting on a Hilbert space $\H$,
and fix a faithful normal semifinite weight $\phi$. We denote the modular automorphism group of $\phi$ by $\sigma^\phi_t$. 
Then as $\phi$ is $\s^\phi_t$ invariant,
we see that for all $T\in{\rm dom}\,\phi\subset \cn$ and $t\in\R$
$$
\phi(T)=\phi(\s^\phi_t(T)).
$$
Suppose further that 
the modular group $\s^\phi_t$, which is inner since $\cn$ is semifinite,  is periodic,
and let $\alpha$ be the (least) period of $\s^\phi_t$. Then 
$$
\phi(T)=\frac{1}{\alpha}\int_0^\alpha \phi(\s^\phi_t(T))dt
=\phi\left(\frac{1}{\alpha}\int_0^\alpha \s^\phi_t(T)dt\right)=:(\phi\circ \Psi)(T),
$$
where $\Psi:\cn\to\cm:=\cn^{\s^{\phi}}$ is the expectation onto the 
fixed point algebra $\cm$ defined by the integral. Then the restriction of $\phi$ to $\cm$
is a faithful normal trace. The restriction of $\phi$ to $\cm$ is also semifinite if and only if
$\phi$ is {\em strictly} semifinite, meaning that $\phi$ is the sum of normal positive linear 
functionals whose supports are mutually orthogonal, \cite[p 105]{T}. In everything that follows, we suppose 
that $\phi$ is strictly semifinite.

Given a faithful normal semifinite trace $\tau$ on a von Neumann algebra $\cn$, we define the ideal of $\tau$-compact operators $\K(\cn,\tau)$
 to be the norm closure of the ideal generated by the projections $p$ with finite trace, $\tau(p)<\infty$.

\begin{defn}
Let $\cn$ be a semifinite von Neumann algebra acting on a Hilbert space $\H$,
and fix a faithful normal strictly semifinite weight $\phi$. Suppose further that 
the modular group $\s^\phi_t$ is periodic. Then we say that $(\A,\H,\D)$ is a 
{\bf unital modular spectral triple} with respect to $(\cn,\phi)$ if

0) $\A$ is a separable unital $*$-subalgebra of $\cn$ with norm closure $A$; 

1) $\A$ is invariant under $\s^\phi$, $\A$ consists of analytic vectors for $\s^\phi$, and $\s^\phi|_{A}$ is a strongly continuous action;

2) $\D$ is a self-adjoint operator affiliated to the fixed point algebra
$\cm:=\cn^{\s^\phi}$;

3) $[\D,a]$ extends to a bounded operator in $\cn$ for all $a\in\A$;

4) $(1+\D^2)^{-1/2}\in \K(\cm,\phi|_\cm)$.

The triple is even if there exists $\gamma=\gamma^*\in\cm$ with $\gamma^2=1$, 
$\gamma a=a\gamma$ for all $a\in\A$ and $\gamma\D+\D\gamma=0$. Otherwise
the triple is odd.

We say that the triple is finitely summable with spectral dimension $p\geq 1$ if $p$
is the least number such that
$$
\phi((1+\D^2)^{-s/2})<\infty\ \ \ \ \mbox{for all} \ \ \Re(s)>p.
$$
\end{defn}

Just as for ordinary spectral triples, there is a notion of smoothness and pseudodifferential operators for $QC^\infty$
modular spectral triples, just as  in \cite{CPRS2, CM}, which we recall here.

\begin{defn}\label{qck} A modular spectral triple $(\A,\H,\D)$ relative to $(\cn,\phi)$ is $QC^k$ 
for $k\geq 1$ 
($Q$ for quantum) if for all $a\in\A$ 
the operators $a$ and $[\D,a]$ are in the domain of $\delta^k_1$, where 
$\delta_1(T)=[(1+\D^2)^{1/2},T]$ is the partial derivation on $\cn$ defined by $(1+\D^2)^{1/2}$. We 
say that 
$(\A,\H,\D)$ is $QC^\infty$ if it is $QC^k$ for all $k\geq 1$. Equivalently, \cite[Proposition 6.5]{CPRS2} and \cite[Lemma B.2]{CM},
$(\A,\H,\D)$ is $QC^\infty$ if for all $a\in\A$ we have
$a,\,[\D,a]\in\bigcap_{k,l\geq 0}{\rm dom}L_1^k\circ R_1^l$, where $L,\,R$ are defined by
$$
L(T)=(1+\D^2)^{-1/2}[\D^2,T]\quad{\rm and}\quad R(T)=[\D^2,T](1+\D^2)^{-1/2}.
$$
\end{defn}

\begin{defn} Let $(\A,\H,\D)$ be a modular spectral triple  relative to $(\cn,\phi)$. For $r\in{\R}$
\begin{equation*} 
{\rm OP}^r=(1+\D^2)^{r/2}\left(\bigcap_{n\geq 0}{\rm dom}\,\delta_1^n\right).
\end{equation*}
If $T\in {\rm OP}^r$ we say that $T$ is a pseudodifferential operator and that the order of $T$ is (at most) $r$.
The definition is actually symmetric, since for $r$ an integer 
(at least) we have by \cite[Lemma 6.2]{CPRS2} 
\begin{align*} 
{\rm OP}^r&=(1+\D^2)^{r/2}\left(\bigcap\mbox{dom}\,\delta_1^n\right)
=(1+\D^2)^{r/2}\left(\bigcap\mbox{dom}\,\delta_1^n\right)
(1+\D^2)^{-r/2}(1+\D^2)^{r/2}\\
&\subseteq\left(\bigcap\mbox{dom}\,\delta_1^n\right)(1+\D^2)^{r/2}.
\end{align*}
From this we easily see that ${\rm OP}^r\cdot {\rm OP}^s\subseteq {\rm OP}^{r+s}$.
Finally, we note that if $b\in {\rm OP}^r$ for $r\geq 0$, then since $b=(1+\D^2)^{r/2}a$ 
for some $a\in {\rm OP}^0$, we get
$[(1+\D^2)^{1/2},b]=(1+\D^2)^{r/2}[(1+\D^2)^{1/2},a]=(1+\D^2)^{r/2}\delta_1(a),$
so $[(1+\D^2)^{1/2},b]\in {\rm OP}^r$.
\end{defn}

{\bf Remarks}: 1) An operator $T\in {\rm OP}^r$ if and only if 
$(1+\D^2)^{-r/2}T\in\bigcap_{n\geq 0}\mbox{dom}\,\delta_1^n$. 
Observe that operators of order at most zero are 
bounded. 

2) We will need a weaker notion of smoothness, introduced in Section 3, for modular spectral triples, as 
Definition \ref{qck} is not satisfied for our main example, the Podle\'s sphere.

{\bf Example.} A semifinite spectral triple is a modular spectral triple with $\phi$ a semifinite normal trace
(and so  $\cm=\cn$). 

{\bf Example.} Given a circle action on a unital $C^*$-algebra $A$, every state on $A$ which is KMS
for this circle action gives rise to a modular spectral triple of dimension 1. Explicit examples
are the Cuntz algebra with its usual gauge action, \cite{CPR2}, and the quantum group $SU_q(2)$
with its Haar state, \cite{CRT}. All these examples are $QC^\infty$ (or regular or smooth) when we use the algebra
of analytic vectors $\A\subset A$ for the circle action. More examples arising from a topological
version of the group-measure space construction are presented in \cite{CPPR}.

{\bf Example.} The only other unital example (known to the authors) is the Podle\'s sphere, which provides
an example of a modular spectral triple of dimension 2. This was first presented in \cite{DS}, and 
has been studied in numerous subsequent works by various authors. The paper \cite{W} provides
a good summary. This example is not $QC^\infty$, but a replacement for the pseudodifferential
calculus was developed in \cite{NT}. This example is our main motivation for weakening the $QC^\infty$
condition, and this example will be presented in detail in Section 4.

{\bf Nonunital examples.} We have chosen to work in the unital case for simplicity, but there
are nonunital examples, \cite{CNNR,CMR}. However, to simplify the discussion of the local index formula,
we will restrict to the unital  case. To handle the nonunital case in general, we would need to
modify the definition of modular spectral triple in order to utilise (analogues of) the results of \cite{CGRS2}, where
the local index formula is proved in the nonunital case.

\subsection{Equivariant $KK$-theory and modular spectral triples.}
\label{sec:spec-flow}

An odd modular spectral triple $(\A,\H,\D)$ with respect to $(\cn,\phi)$ 
defines an equivariant Kasparov module, and so a 
class $[\D]\in KK^{1,\T}(A,\K_\cn)$, where we recall that $A=\overline{\A}$. The construction of the Kasparov module 
associated to a modular spectral triple begins with the definition of a suitable ideal.
We will deal explicitly with the odd case here, 
just stating the analogous results in the even case.

\begin{defn}
\label{defn_j_phi}
Given a modular spectral triple $(\cA,\cH,\cD,\cN,\phi)$, let
\enveqn{
J_{\phi} := \{ SkT \colon S, \, T \in \cN, \ k \in \cK(\cM, \phi|_{\cM}) \}
}
denote the norm closed two-sided ideal in $\cN$ generated by $\cK(\cM, \phi|_{\cM})$.
\end{defn}

The ideal $J_{\phi}$ is a right Hilbert module over itself, and $\cA$ acts on the left of $J_{\phi}$ by multiplication. 
The axioms of a modular spectral triple imply that $(1 + \D^{2})^{-1/2} \in J_{\phi}$. With a little effort we can show, as in \cite[Theorem 4.1]{KNR}, that
the pair $(J_\phi,\cD(1+\cD^2)^{-1/2})$ is a Kasparov module, except that the module $J_\phi$ may not be countably generated. 

To deal with this problem, we recall the following construction from \cite[Theorem 5.3]{KNR}.

\begin{defn}
\label{def:bee-phi}
Let $(\cA,\cH,\cD,\cN,\phi)$ be a modular spectral triple, where we recall that $\A$ is separable.
Write $F_\D:=\D(1+\D^2)^{-1/2}$ and let $B_\phi$ be the smallest $C^*$-algebra in $\cN$
containing the elements
\[
F_\cD\,[F_\cD,a], \quad b\,[F_\cD,a], \quad
F_\cD\, b\,[F_\cD,a], \quad  a\,\varphi(\cD)
\]
for all $a,b\in \mathcal{A}$ and $\varphi\in C_0(\R)$. Then
$B_\phi$ 
is separable, and so $\sigma$-unital, and
contained in $J_\phi$. 
\end{defn}

\begin{prop}
\label{prop:Kas-mod}
A modular spectral triple $(\cA,\cH,\cD,\cN,\phi)$ defines an equivariant $KK$-theory class $[\D] = [B_\phi, F_{\D}] \in KK^{1, \bT}(A, B_\phi)$,
where $F_\D:=\D(1+\D^2)^{-1/2}$.
\end{prop}
\begin{proof}
A modular spectral triple is automatically a von Neumann spectral triple with respect to $J_\phi$ in the sense of \cite{KNR}.
Then  \cite[Theorem 5.3]{KNR} shows that $B_\phi$ is a countably generated right $C^*$ $B_\phi$-module, and that the
pair $(B_\phi,F_\cD)$ is a Kasparov module. The equivariance is immediate.
\end{proof}

Having obtained an equivariant Kasparov module, and so a $KK$-class, the Kasparov product defines a $K_0^{\bT}(B_\phi)$-valued 
index pairing between a modular spectral triple and equivariant $K$-theory. That is,
\enveqn{
K_{1}^{\bT}(A) \times KK^{1, \bT}(A, B_{\phi}) \rightarrow K_{0}^{\bT}(B_{\phi}).
}
See \cite[Theorem 18.4.4]{B} for example.
We now seek an analytic formula to compute this index, and in Section \ref{sec:local-index} we obtain such a formula, the resolvent index formula. 
The first step is the construction of a semifinite spectral triple which encodes the index pairing between a 
modular spectral triple and an equivariant $K$-theory class. This is necessary to obtain a well-defined 
numerical index problem. We now describe this procedure.

Given a modular spectral triple $(\A,\H,\D,\cn,\phi)$ and a class $[u]\in K_1^\T(A)$, there is a unitary $u\in M_n(\A)$ and a representation
$V:\T\to M_n(\C)$ such that $u$ is $\sigma^\phi\otimes Ad\,V$ invariant, \cite{B,CNNR}. 
In particular, if $n=1$ then $u$ is $\s^\phi$ invariant.

We can diagonalise the representation $V_t=\oplus_{j=1}^n \lambda_j^{it}$, $\lambda_j\in[1,\infty)$, and 
in this basis it is clear that

1) $u_{ij}$ transforms under $Ad\,V_t$ by $\lambda_i^{it}\lambda_j^{-it}$;

2) $u_{ij}$ transforms under $\s_t^\phi$ by $\lambda_i^{-it}\lambda_j^{it}$;

3) $V_t$ extends to an action of $\C$ which is not a $*$-action, but satisfies
$V_z^*=V_{-\bar{z}}$.

We define a positive functional $G:M_n(\C)\to \C$ by setting
$$
G(T)=\Tr(V_{-i}T),\ \ \ T\in M_n(\C).
$$
Then $G$ is a $KMS_{1}$ functional on $M_n(\C)$, \cite{BR}, for the action $Ad\,V$, but is not a state 
as it is not normalised.

Now consider the fixed point algebra 
$\cm_n=(M_n(\cn))^{\s^\phi\otimes Ad\,V}$, which is the centralizer of 
the weight $\phi\otimes G$, \cite[Proposition 4.3]{T}. Then $\phi\otimes G$ restricts to a 
faithful normal semifinite trace on
$\cm$ and moreover $u \in \cm_n$. The latter statement follows from the 
definition of $u$. The former follows since the strict semifiniteness of $\phi$ implies the
strict semifiniteness of $\phi\otimes G$.

\begin{lemma}
\label{lem:bob-the-builder}
Let  $(\A,\H,\D,\cN,\phi)$ be a modular
spectral triple which is  finitely summable and $u\in M_n(\A)$ a $\sigma^\phi$ equivariant unitary, with
associated representation $V:\T\to M_n(\C)$. Then 
$$
(C^\infty(u),\H\otimes \C^n,\D\otimes {\rm Id}_n, \cm_n,\phi\otimes G)
$$ 
is a finitely summable  semifinite spectral triple.
Here $C^\infty(u)$ is the algebra of all $f(u)\in C^*(u)$ with $f$ a $C^\infty$ function on
the spectrum of $u$. Let $B_{\phi\otimes G}\subseteq \K(\cm_n,\phi\otimes G)$ be defined as in Definition \ref{def:bee-phi}. Then 
this semifinite spectral triple defines a 
Kasparov class in $KK^{1,\T}(C^*(u),B_{\phi\otimes G})$.
\end{lemma}
\begin{proof}
The statement that we obtain a semifinite spectral triple follows from the construction, and that we get a Kasparov
module follows from Proposition \ref{prop:Kas-mod}. 
\end{proof}

Thus given $[u,V]\in K_1^\T(\A)$, we apply \cite[Theorem 6.9]{KNR} to compute the spectral flow, \cite{Ph}. Let 
$i:B_{\phi\otimes G}\subset \K(\cm_n,\phi\otimes G)$ be the inclusion, and $i_*:K_0(B_{\phi\otimes G})\to K_0(\K(\cm_n,\phi\otimes G))$.
Then \cite[Theorem 6.9]{KNR} allows us to compute the spectral flow as
\enveq{
\label{eqn_sf}
sf_{\phi\otimes G}(\D\ox {\rm Id}_n,\,u(\D\ox {\rm Id}_n)u^*) = (\phi\otimes G)_*(i_*([u]\otimes_{C^*(u)}[\D\otimes{\rm Id}_n])).
}

At this point, we have obtained an index problem which {\em a priori} depends on the representative 
$u$ of the equivariant $K$-theory class (through the use of $C^*(u)$). To show that we do indeed have a well-defined pairing with $K_1^\T(\A)$,
we will show, via the resolvent index formula, that the index can be computed in terms of the Chern character of $u$, which is independent of
the chosen representative of the class $[u]$. Finally, we show that the original index pairing between a modular 
spectral triple and equivariant $K$-theory can be described by the spectral flow above.


\section{The resolvent index formula in twisted cyclic cohomology}
\label{sec:local-index}

In this section we express the spectral flow from Equation \eqref{eqn_sf} in terms of the pairing between a twisted 
cyclic cocycle dependent only on the modular spectral triple and the Chern character of the equivariant unitary. 
In order to achieve this without invoking the $QC^\infty$ property, we make a technical improvement on the 
work of \cite{CPRS2} by using a  weaker smoothness condition. This is necessary for our application, as the Podle\'s sphere modular spectral triple is not
$QC^\infty$.

\subsection{Weakly $QC^\infty$ modular spectral triples}
\label{subsec:weak-qcinfty}

We weaken the $QC^\infty$ condition with the aim of justifying a
resolvent expansion, used in the proof of our index formulae, without recourse to the
pseudodifferential calculus. There are two basic reasons for doing this. 

The first is that
the example of the Podle\'s sphere shows 
that we do not always have the $QC^\infty$ property for modular
spectral triples. 

The  second reason is that, conceptually, the use of the pseudodifferential calculus to
prove existence and continuity of the resolvent cocycle  is 
overkill, requiring us to invoke much more smoothness than is necessary for the 
statment of existence and continuity. 

\begin{defn} 
\label{defn:weak-op}
Let $(\A,\H,\D)$ be a modular spectral triple relative to $(\cn,\phi)$.
For $T\in\cn$ mapping the domain of $\D^2$ to itself, define 
\begin{align*}
WL(T):=(1+\D^2)^{-1}[\D^2,T]&=(1+\D^2)^{-1}T(1+\D^2)-T,\\ 
WR(T):=[\D^2,T](1+\D^2)^{-1}&=(1+\D^2)T(1+\D^2)^{-1}-T.
\end{align*}
We say that $(\A,\H,\D)$ is weakly $QC^\infty$ if 
$$
\A\subset {\rm OP}^0\subset \cN\quad{\rm and}\quad
[\D,\A]\subset \wop^0:=\bigcap_{k,l\geq 0}\mbox{dom}(WL)^k(WR)^l\subset \cN.
$$
\end{defn}

The analogous definition of weak $QC^k$ is awkward, since in Definition \ref{qck}, $QC^k$ is defined in
terms of commutators with $|\D|$ or $(1+\D^2)^{1/2}$. We will leave aside these questions,
and just work with weak $QC^\infty$. Also, $QC^\infty$ implies weak $QC^\infty$ 
by the boundedness of $(1+\D^2)^{-1/2}$.

While we do not have a pseudodifferential calculus for a weakly $QC^\infty$ modular
spectral triple $(\A,\H,\D)$, we may consider the 
weak pseudodifferential operators of order $s\in\R$ given by
$$
\wop^s:=(1+\D^2)^{s/2}\left(\bigcap_{k,l}\,{\rm dom}\,WL^k\circ WR^l\right).
$$
This definition is symmetric, in the sense that 
$$
\wop^s=\left(\bigcap_{k,l}\,{\rm dom}\,WL^k\circ WR^l\right)(1+\D^2)^{s/2},
$$
since for all $s\in\R$, $\wop^s$ is preserved 
by $T\mapsto (1+\D^2)^{\pm s}T(1+\D^2)^{\mp s}$,
by Lemma \ref{lem:bdd-conj} below.
Observe also that we have ${\rm OP}^s\subset \wop^s$.

It follows from the definitions that if $(\A,\H,\D)$
is a weakly $QC^\infty$ modular spectral triple and $u\in M_n(\A)$ is an equivariant unitary, then the associated
semifinite spectral triple $(C^\infty(u),\H\otimes\C^n,\D\otimes{\rm Id}_n, \cm_n,\phi\otimes G)$ 
is also weakly $QC^\infty$.

The next few lemmas record some basic properties of the maps $WL$ and $WR$.

\begin{lemma} 
\label{lem:bdd-conj}
Let $\D:{\rm dom}\D\subset \H\to \H$ be an unbounded self-ajoint operator. 
Then $T\in \B(\H)$ belongs to
$$
\bigcap_{k,l\geq 0}{\rm dom}(WL)^k(WR)^l
$$
if and only if $(1+\D^2)^{s/2}T(1+\D^2)^{-s/2}$ extends to a bounded operator for all $s\in\R$.
\end{lemma}

\begin{proof}
It follows from Definition \ref{defn:weak-op} that $T\in  \bigcap_{k,l\geq 0}\mbox{dom}(WL)^k(WR)^l$ if and only if 
$(1+\D^2)^{k}T(1+\D^2)^{-k}$ is a bounded operator for all 
$k\in\Z$. It is also immediate that if $(1+\D^2)^{s/2}T(1+\D^2)^{-s/2}$ is 
bounded for all $s\in\R$, then $T\in \bigcap_{k,l\geq 0}\mbox{dom}(WL)^k(WR)^l$. 
So let $0<s<1$, and recall that
$$
(1+\D^2)^{-s}=\frac{\sin(s\pi)}{\pi}\int_0^\infty \lambda^{-s}(1+\lambda+\D^2)^{-1}d\lambda.
$$
Then for $T\in \bigcap_{k,l\geq 0}\mbox{dom}(WL)^k(WR)^l$ we have
\begin{align*}
&(1+\D^2)^sT(1+\D^2)^{-s}\\
&=(1+\D^2)^sT\frac{\sin(s\pi)}{\pi}\int_0^\infty \lambda^{-s}(1+\lambda+\D^2)^{-1}d\lambda\\
&=(1+\D^2)^s\frac{\sin(s\pi)}{\pi}\int_0^\infty \lambda^{-s}
\left((1+\lambda+\D^2)^{-1}T+(1+\lambda+\D^2)^{-1}[\D^2,T](1+\lambda+\D^2)^{-1}\right)d\lambda\\
&=(1+\D^2)^s\frac{\sin(s\pi)}{\pi}\int_0^\infty \lambda^{-s}
(1+\lambda+\D^2)^{-1}\left(T+[\D^2,T](1+\D^2)^{-1}(1+\D^2)(1+\lambda+\D^2)^{-1}\right)d\lambda\\
&=T+(1+\D^2)^s\frac{\sin(s\pi)}{\pi}\int_0^\infty \lambda^{-s}
(1+\lambda+\D^2)^{-1}[\D^2,T](1+\D^2)^{-1}(1+\D^2)(1+\lambda+\D^2)^{-1}d\lambda.
\end{align*}
An application of the functional calculus now
shows that the integral is norm convergent, but in order to show that $(1+\D^2)^s$ times the integral is bounded, we must work a little harder.
We write
$$
(1+\D^2)(1+\lambda+\D^2)^{-1}=1-\lambda(1+\lambda+\D^2)^{-1},
$$
so that the integral can be written, with $B=[\D^2,T](1+\D^2)^{-1}$, as
\begin{align*}
&\int_0^\infty \lambda^{-s}
(1+\lambda+\D^2)^{-1}[\D^2,T](1+\D^2)^{-1}(1+\D^2)(1+\lambda+\D^2)^{-1}d\lambda\\
&=\int_0^\infty \lambda^{-s}
(1+\lambda+\D^2)^{-1}B\,d\lambda-\int_0^\infty \lambda^{-s}\,\lambda
(1+\lambda+\D^2)^{-1}B(1+\lambda+\D^2)^{-1}\,d\lambda.
\end{align*}
The first integral on the right hand side converges in norm to $\frac{\pi}{\sin(s\pi)}(1+\D^2)^{-s}B$. For the second integral on the
right hand side, we suppose first that $B$ is self-adjoint. Then 
$$
\lambda(1+\lambda+\D^2)^{-1}B(1+\lambda+\D^2)^{-1}\leq \Vert B\Vert\,\lambda(1+\lambda+\D^2)^{-2}\leq \Vert B\Vert\,(1+\lambda+\D^2)^{-1}.
$$
Thus for $B$ self-adjoint, the second integral converges in norm to an operator which is bounded above by $\frac{\pi}{\sin(s\pi)}(1+\D^2)^{-s}\Vert B\Vert$. By decomposing 
$B$ into its real and imaginary parts, this is true for any bounded $B$.
Thus  for $0<s<1$, $(1+\D^2)^sT(1+\D^2)^{-s}$ is bounded, and a
similar argument shows that $(1+\D^2)^{-s}T(1+\D^2)^{s}$ is bounded.
\end{proof}

In all the following, we define $R_s(\lambda):=(\lambda-(1+s^2+\D^2))^{-1}$ for $s\geq 0$ and 
$\lambda$ in the vertical line 
$$
l:=\{a+iv: \ -\infty<v<\infty\}
$$
for some fixed $0<a<1/2$.

\begin{lemma}
\label{lem:bdd}
Let $(\A,\H,\D)$ be a weakly $QC^\infty$ modular spectral triple relative to
$(\cn,\phi)$. Then $R_s(\lambda)[\D^2,T]$ is 
uniformly bounded on the line $l$ independent
of $s,\,\lambda$ for all $T\in\A\cup[\D,\A]$. For all $T\in\A\cup[\D,\A]$, the function 
$\lambda\mapsto R_s(\lambda)TR_s(\lambda)^{-1}$ is uniformly bounded
and differentiable on the line $l$ with derivative $-R_s(\lambda)^2[\D^2,T]$ which vanishes 
as $\lambda\to a\pm i\infty$.
\end{lemma}
\begin{proof}
First $R_s(\lambda)[\D^2,T]=R_s(\lambda)(1+\D^2)(1+\D^2)^{-1}[\D^2,T]$ and 
$R_s(\lambda)(1+\D^2)$ is uniformly bounded. Then $R_s(\lambda)TR_s(\lambda)^{-1}=R_s(\lambda)[\D^2,T]+T$
is uniformly bounded on $l$.
For the differentiability, we form the difference quotients where $\epsilon$ is chosen so that $\lambda+\epsilon$ 
lies in a small ball centred on $\lambda=a+iv$
\begin{align*}
&R_s(\lambda+\epsilon)TR_s(\lambda+\epsilon)^{-1}-R_s(\lambda)TR_s(\lambda)^{-1}\\
&=(R_s(\lambda+\epsilon)-R_s(\lambda))TR_s(\lambda+\epsilon)^{-1}
+R_s(\lambda)T(R_s(\lambda+\epsilon)-R_s(\lambda)^{-1})\\
&=-\epsilon R_s(\lambda+\epsilon)R_s(\lambda)TR_s(\lambda+\epsilon)^{-1}
+\epsilon R_s(\lambda)T.
\end{align*}
Now the uniform boundedness of $R_s(\lambda)TR_s(\lambda)^{-1}$ 
and the boundedness of $R_s(\lambda)T$ show that
after dividing by $\epsilon$, the norm limit as $\epsilon\to 0$ exists and is given by
$$
R_s(\lambda)T-R_s(\lambda)^2TR_s(\lambda)^{-1}=-R_s(\lambda)^2[\D^2,T].
$$
This is not only bounded, but goes to zero as $|\lambda|\to \infty$ along the line $l=a+iv$.
\end{proof}

\begin{lemma}
\label{lemma_bdd_diff_resolvent}
With $(\A,\H,\D)$ as above and $T\in\A\cup[\D,\A]$, we have the formula
$$
R_s(\lambda)^nTR_s(\lambda)^{-n}=T+\sum_{j=1}^n(n-j+1)R_s(\lambda)^jT^{(j)}.
$$
\end{lemma}
\begin{proof}
Induction and the easy formula $R_s(\lambda)TR_s(\lambda)^{-1}=T+R_s(\lambda)[\D^2,T]$.
\end{proof}


\begin{corl} 
\label{cr:deriv}
The function $\lambda\mapsto R_s(\lambda)^nTR_s(\lambda)^{-n}$ 
is norm differentiable for all $T\in\A\cup[\D,\A]$.
The derivative goes to zero in norm as $\lambda\to a\pm i\infty$ and is given by
\enveq{
\label{eqn_ddl_rtr}
\frac{d}{d\lambda}R_s(\lambda)^nTR_s(\lambda)^{-n}
=-R_s(\lambda)\sum_{j=1}^nj(n-j+1)R_s(\lambda)^jT^{(j)}.
}
\end{corl}

We now prove the main technical result we require, which weakens the smoothness hypotheses of \cite[Lemma 7.2]{CPRS2}.

\begin{lemma}
\label{lemma_Bs_trace}
Let $(\cA, \cH, \cD, \cN, \phi)$ be a weakly $QC^\infty$ modular spectral triple of dimension $p \geq 1$. Let $m$ be a non-negative integer and $j = 0, \dots, m$.
\begin{enumerate}
\item Let $A_j \in \wop^{k_j}$, $k_j \geq 0$, with the product $A_{0} A_{1} \cdots A_{m}$ being $\sigma^{\phi}_{t}$-invariant and
affiliated to $\cM=\cN^{\sigma^\phi}$. Then the map
\enveqn{
r \mapsto B^{r}(s) = \frac{1}{2\pi i} \int_{l} \lambda^{-p/2-r}A_0 R_s(\lambda)A_1 R_s(\lambda)
A_2 \cdots R_s(\lambda) A_m R_s(\lambda) d\lambda,
}
is an analytic function with values in $\mathrm{Dom}(\phi)$ for $r \in \{z \in \bC:\ \Re(z)>|k|/2 - m,\ z\not\in\bN-p/2\}$, where 
$|k| = k_0 + k_1 + \cdots + k_m$. For $\alpha > 0$, the function $s \mapsto s^{\alpha} \times \phi \big( |B^{r}(s)| \big)$ is 
integrable on $[0,\infty)$ when in addition we have $1+\alpha+|k|-2m<2\Re(r)$.
\item Define $\hat{R}_s(\lambda) = (\lambda-(1 + s^{2} + \cD^{2} + sK))^{-1}$, for an operator $K = K^{\ast}$ with 
$\Vert K \Vert_{\infty} \leq \sqrt{2}$. For $a_j \in \cA$, with $a_{0} a_{1} \cdots a_{m} \in \cM$, and $r \in \bC$, with $\Re(r)>0$, the operator
\enveqn{
\tilde{B}^{r}(s) = \frac{1}{2\pi i}\int_l\lambda^{-p/2-r}a_0R_s(\lambda)[\cD, a_1]
R_s(\lambda)[\cD, a_2]\cdots R_s(\lambda) [\cD, a_m] \hat{R}_s(\lambda) d\lambda
}
is in $\mathrm{Dom}(\phi)$, and the function $s \mapsto s^{m} \times \phi \big( |\tilde{B}^{r}(s)| \big)$ is integrable on $[0,\infty)$ when 
$p < 1 + m$ and $1 < m + 2\Re(r)$.
\end{enumerate}
\end{lemma}
\begin{proof}
The restriction of $\phi$ to the fixed point algebra $\cM := \cN^{\sigma^{\phi}}$ is a semifinite trace. 
By assumption, we have $(1 + \cD^{2})^{-1/2} \in \cM$, so $R_{s}(\lambda) \in \cM$, and $A_{0} A_{1} \cdots A_{m}$ 
is affiliated to $\cM$. Hence, the estimates in this proof will be done in the von Neumann algebra $\cM$, 
and we denote the trace norm, with respect to $\phi$ on $\cM$, by $\Vert \cdot \Vert_{1}$.


To prove statement 1, the strategy is to use the fundamental theorem of calculus, at first just doing norm convergence
of integrals and norm differentiability. 
We abbreviate $R := R_s(\lambda)$, fix $k_0,\dots,k_m$ as in the statement, and with $\Re(r)>0$ sufficiently large, we have for any integer $M>m$
\enval{
&\frac{1}{2\pi i} \int_l \lambda^{-p/2-r}A_0 R A_1 R \cdots R A_m R d\lambda \nonumber \\
& \qquad = \frac{1}{2\pi i} \int_l 
\lambda^{-p/2-r}A_0RA_1R^{-1}R^2A_2R^{-2}\cdots R^mA_mR^{-m}R^{m+1}d\lambda \nonumber \\
&\qquad = \frac{(-1)^{M-m}}{2\pi i(p/2+r-1)(p/2+r-2)\cdots(p/2+r-(M-m))}\times \nonumber \\
& \qquad\qquad \times \int_l \frac{d^{M-m}}{d\lambda^{M-m}}
\left(\lambda^{-p/2-r+(M-m)}\right)
A_0RA_1R^{-1}R^2A_2R^{-2}\cdots R^mA_mR^{-m}R^{m+1}d\lambda \nonumber \\
& \qquad = \frac{\Gamma(p/2 + r - (M-m))}{2\pi i \, \Gamma(p/2+r)} \int_l \lambda^{-p/2-r+(M-m)} \sum_{j=0}^{M-m} \times \nonumber \\
& \qquad \qquad \times \bca M-m \\ j \eca \frac{d^j}{d\lambda^j}\!
\left(A_0RA_1R^{-1}R^2A_2R^{-2}\cdots R^mA_mR^{-m}\right)\!\!\frac{d^{M-m-j}}{d\lambda^{M-m-j}}(R^{m+1})d\lambda. \label{eqn_ch3_ftc_derivs}
}

Iterating the derivative $\frac{d}{d\lambda}(R^{m+1}) = -(m+1)R^{m+2}$ yields
\enveqn{
\frac{d^{M-m-j}}{d\lambda^{M-m-j}}(R^{m+1}) = (-1)^{M-m-j} \left( \prod_{n = m+1}^{M+1-j} n \right) R^{M+1-j}.
}
Now we consider
\enveqn{
\frac{d^j}{d\lambda^j} \left(A_0 R A_1 R^{-1} R^2 A_2 R^{-2} \cdots R^m A_m R^{-m} \right).
}
We would like to apply Lemma \ref{lem:bdd} to this term, however recall that each $A_{j} \in \wop^{k_{j}}$ and not $\wop^{0}$. So, we rewrite
\envaln{
&A_0 R A_1 R^{-1} R^2 A_2 R^{-2} \cdots R^m A_m R^{-m} \\
&\qquad = A_0 (1 + \cD^{2})^{-k_{0}/2} \Big(R (1 + \cD^{2})^{k_{0}/2} A_1 (1 + \cD^{2})^{-(k_{0}+k_{1})/2} R^{-1} \Big) \times \\
& \qquad \qquad \times \Big( R^2 (1 + \cD^{2})^{(k_{0}+k_{1})/2} A_2 (1 + \cD^{2})^{-(k_{0}+k_{1}+k_{2})/2} R^{-2} \Big) \times \cdots \times \\
& \qquad \qquad \qquad \times \Big( R^m (1 + \cD^{2})^{(|k|-k_{m})/2} A_m (1 + \cD^{2})^{-|k|/2} R^{-m} \Big) (1 + \cD^{2})^{|k|/2}.
}
By definition we have $A_{j} (1 + \cD^{2})^{-k_{j}/2} \in \wop^{0}$, so using Lemma \ref{lem:bdd-conj}, 
we now find that $(1 + \cD^{2})^{s} A_{j} (1 + \cD^{2})^{-s-k_{j}/2} \in \wop^{0}$ for all $s \in \bR$. Hence, we define
\enveqn{
A'_{j} := (1 + \cD^{2})^{\half \sum_{n = 0}^{j-1} k_{n}} \, A_{j} \, (1 + \cD^{2})^{- \half \sum_{n = 0}^{j} k_{n}} \in \wop^{0},
}
so that
\enveqn{
A_0 R A_1 R^{-1} R^2 A_2 R^{-2} \cdots R^m A_m R^{-m} = A'_0 R A'_1 R^{-1} R^2 A'_2 R^{-2} \cdots R^m A'_m R^{-m} (1 + \cD^{2})^{|k|/2}.
}
The purpose of introducing $A'_j$ is to move all the $\wop^{k_{j}}$ behaviour into the factor $(1 + \cD^{2})^{|k|/2}$ on the right.

We now invoke Corollary \ref{cr:deriv}, and find that each factor $R^{j} A'_{j} R^{-j}$ is norm differentiable with respect to $\lambda$. 
Indeed, by Equation \eqref{eqn_ddl_rtr} we have 
\enveq{
\label{eqn_ch3_diff_rn_bdd}
\frac{d^n}{d\lambda^n} R^{j} A'_{j} R^{-j} = R^{n} B(\lambda)
}
for $n \geq 0$, and some operator $B(\lambda)$ uniformly bounded in $s, \, \lambda$. So, we apply the chain rule to 
\envaln{
&\frac{d^j}{d\lambda^j} \left(A'_0 R A'_1 R^{-1} R^2 A'_2 R^{-2} \cdots R^m A'_m R^{-m} \right) (1 + \cD^{2})^{|k|/2} \\
& \qquad \qquad \qquad \qquad \qquad \qquad \qquad \qquad = \frac{d^j}{d\lambda^j} \left(A_0 R A_1 R^{-1} R^2 A_2 R^{-2} \cdots R^m A_m R^{-m} \right),
}
and use Equation \eqref{eqn_ch3_diff_rn_bdd} to compute the derivatives. Then Lemma \ref{lemma_bdd_diff_resolvent} 
allows us to move each resolvent $R^{n}$ arising from Equation \eqref{eqn_ch3_diff_rn_bdd} to the left, which gives
\enveqn{
\frac{d^j}{d\lambda^j} \left(A_0 R A_1 R^{-1} R^2 A_2 R^{-2} \cdots R^m A_m R^{-m} \right) = R^{j} B_{j}(A'_{0}, \dots, A'_{m}) (1 + \cD^{2})^{|k|/2},
}
where $B_{j}(A'_{0}, \dots, A'_{m}) \in \wop^{0}$ is uniformly bounded in $s, \, \lambda$.

We absorb the constants $(-1)^{M-m-j} \left( \prod_{n = m+1}^{M+1-j} n \right)$ into $B_{j}(A'_{0}, \dots, A'_{m})$, and 
apply the derivative computations to Equation \eqref{eqn_ch3_ftc_derivs}, which yields
\envaln{
& \frac{1}{2\pi i} \int_l \lambda^{-p/2-r}A_0 R A_1 R \cdots R A_m R d\lambda \\
&\qquad \qquad = \frac{\Gamma(p/2 + r - (M-m))}{2\pi i \, \Gamma(p/2+r)} \int_l \lambda^{-p/2-r+(M-m)} \sum_{j=0}^{M-m} \times\\
&\qquad \qquad \qquad \quad \times \bca M-m \\ j \eca R^{j} B_{j}(A'_{0}, \dots, A'_{m}) (1 + \cD^{2})^{|k|/2} R^{M+1-j} d\lambda\\
&\qquad \qquad = \frac{\Gamma(p/2 + r - (M-m))}{2\pi i \, \Gamma(p/2+r)} \int_l \lambda^{-p/2-r+(M-m)} \sum_{j=0}^{M-m} \times\\
&\qquad \qquad \qquad \quad \times \bca M-m \\ j \eca R^{j} B_{j}(A'_{0}, \dots, A'_{m}) R^{-j} (1 + \cD^{2})^{|k|/2} R^{|k|/2} R^{M+1-|k|/2} d\lambda,
}
where the square roots use the principal branch of $\log$.

For each $j$, the operator $R^{j} B_{j}(A'_{0}, \dots, A'_{m}) R^{-j}$ is uniformly bounded in 
$s,\,\lambda$ by Lemma \ref{lemma_bdd_diff_resolvent} and the uniform boundedness of $B_{j}(A'_{0}, \dots, A'_{m})$. 
Also, the operator $(1 + \cD^{2})^{|k|/2} R^{|k|/2}$ is uniformly bounded in $s, \, \lambda$, so we are left with estimating 
$R^{M+1-|k|/2}$. The trace estimate for the resolvent in \cite[Lemma 5.3]{CPRS2} states that for  $M$ large enough and  all 
$\epsilon>0$, there is a constant $C_\epsilon>0$ such that
\begin{equation}
\Vert R^{M+1-|k|/2}\Vert_1\leq C_\epsilon ((1/2+s^2-a)^2+v^2)^{-(M+1-|k|/2)/2+(p/4+\epsilon)}.
\label{eq:cee-eps}
\end{equation}
This estimate, and the uniform boundedness of each $R^{j} B_{j}(A'_{0}, \dots, A'_{m}) R^{-j}$, implies that
\enveqn{
\frac{1}{2\pi i}\int_l \lambda^{-p/2-r}A_0RA_1R\cdots RA_mRd\lambda \in \mathrm{Dom}(\phi)
}
for $|k|-2m + \epsilon<2\Re(r)$, for all $\epsilon>0$. We may apply this estimate only when $r \neq (M-j)-p/2$ as the prefactor 
\enveqn{
\frac{\Gamma(p/2 + r - (M-m))}{2\pi i \, \Gamma(p/2+r)} = \frac{1}{2\pi i(p/2+r-1)(p/2+r-2)\cdots(p/2+r-(M-m))}
}
has a pole at these points. So now we estimate
\enveqn{
\int_0^\infty s^\alpha\,\phi \left(\frac{1}{2\pi i}\int_l \lambda^{-p/2-r}A_0RA_1R\cdots RA_mRd\lambda\right)ds
}
in trace norm (recall that we regard $\phi$ as a trace on the fixed point algebra $\cM$). The calculations above show that the trace norm is bounded by
\envaln{
& \frac{|\Gamma(p/2 + r - (M-m))|}{2\pi |\Gamma(p/2+r)|} \int_0^\infty s^\alpha\, \int_{-\infty}^\infty \sqrt{a^2+v^2}^{-p/2-\Re(r)+(M-m)} \times \\
& \qquad \qquad \times \sum_{j=0}^{M-m}\binom{M-m}{j} \Vert R^{j}B_{j}(A'_0,A'_1,\dots,A'_m)R^{-j}\Vert_\infty\,\Vert R^{M+1-|k|/2}\Vert_1 dvds\\
&\leq
\sum_{j=0}^{M-m}\binom{M-m}{j}
\frac{|\Gamma(p/2 + r - (M-m))|}{2\pi |\Gamma(p/2+r)|} \,  C_{\epsilon}' \times\\
&\times\int_0^\infty s^\alpha\, \int_{-\infty}^\infty \sqrt{a^2+v^2}^{(M-m)-p/2-\Re(r)}
\sqrt{(1/2+s^2-a)^2+v^2}^{|k|/2-M-1+(p+\epsilon)/2} dvds,
}
where the constant $C_\epsilon'$ incorporates the constant from the estimate in Equation \eqref{eq:cee-eps} and the constant coming from
$\Vert R^{j}B_j(A'_0,A'_1,\dots,A'_m)R^{-j}\Vert_\infty\leq C$.
Now by \cite[Lemma 5.4]{CPRS2}, the double integral converges for 
\enveqn{
(\alpha+|k|-M)+(p+\epsilon-M)<1 \quad \text{and} \quad (\alpha+|k|)-2m+\epsilon-2\Re(r)<-1.
}
The first constraint can always be satisfied by taking $M$ sufficiently large. The second holds
precisely when $\alpha+|k|+1-2m<2\Re(r)$, by choosing $\epsilon$ small enough.

Statement 2 of the lemma is proved just as above, with the extra $\hat{R}_s(\lambda)$
just estimated in operator norm, using \cite[Lemma 5.1]{CPRS2}:
\enveqn{
\Vert \hat{R}_s(\lambda)\Vert_\infty\leq (v^2+(1+s^2-a-s\Vert K \Vert_\infty)^2)^{-1/2},
}
and the general integral estimate \cite[Lemma 5.4]{CPRS2}.
\end{proof}

\subsection{Existence of the resolvent cocycle for weakly $QC^\infty$ modular spectral triples}
\label{subsec:resolvent}

First we explicitly define the resolvent cocycle associated to a modular spectral triple, again just working in the odd case.
The definitions in the even case can be deduced from \cite{CPRS3}.
\begin{defn}
\label{expectation}
Let $(\cA,\cH,\cD,\cN,\phi)$ be a weakly $QC^\infty$ odd modular spectral triple of dimension 
$p \geq 1$. Let $N=[p/2]+1$ be the least integer strictly greater than $p/2$. 
Let $m$ be an odd integer, $1 \leq m \leq 2N-1$, and let  $A_j \in \wop^{k_j}$, $j=0,\dots,m$,  be operators whose product $A_{0} A_{1} \cdots A_{m}$ 
is $\sigma^{\phi}_{t}$-invariant and affiliated to $\cM$. For $2 \Re(r) > (k_{0} + \dots + k_{m}) + 1 - m$, $r\not\in\N-p/2$, define
\enveqn{
\langle A_0,\dots,A_m\rangle_{m,s,r} := \frac{1}{2\pi i} \phi \left( \int_l \lambda^{-p/2-r} A_0R_s(\lambda)A_1\cdots A_mR_s(\lambda)d\lambda\right).
}
The resolvent cocycle $(\Phi^r_m)_{m=1,3,\dots,2N-1}$ is defined to be
\enveqn{
\Phi_m^r(a_0,a_1,\dots,a_m) :=  \frac{-2\,\sqrt{2\pi i}}{\Gamma((m+1)/2)} \int_0^\infty s^m\langle a_0,[\cD,a_1],\dots, [\cD,a_m]\rangle_{m,s,r}ds,
}
for $a_{i} \in \cA$ satisfying $a_{0} a_{1} \cdots a_{m} \in \cM$.
\end{defn}

We  observe that $\Phi_m^r$ is finite for $\Re(r)>(1-m)/2$, by Lemma \ref{lemma_Bs_trace}. 
In this subsection we show that for weakly smooth modular spectral triples, $(\Phi_m^r)_{m=1,\dots,2N-1}$ defines a twisted 
$b,B$ cocycle modulo functions holomorphic in the half-plane $r > (1-p)/2$.

We start by presenting the $s$- and $\lambda$-tricks, which are the main tools needed to
prove continuity of the resolvent cocycle. These tricks appeared in \cite{CPRS2,CPRS3,CPRS4}
without appropriate justification for the convergence of the derivatives in trace norm. In \cite{CGRS2}
the justification was given with the aid of the pseudodifferential calculus. Here we present
a different proof using only the weak $QC^\infty$ hypothesis.

\begin{lemma}[$s$-trick] 
Let $(\A,\H,\D)$ be a weakly $QC^\infty$ odd modular
spectral triple relative to $(\cn,\phi)$ of dimension $p\geq 1$.
For any integers $m\geq 0, k\geq 1$
and 
operators $A_0,\dots,A_m$ with $A_j\in \wop^{k_j}$, and $2\Re(r)>k+2\sum k_j-2m$, $r\not\in \N-p/2$,
we may choose $r$ with $\Re(r)$ sufficiently large such that
\begin{align*} 
k\int_0^\infty s^{k-1}\langle A_0,\dots,A_m\rangle_{m,s,r}ds=
-2\sum_{j=0}^m\int_0^\infty s^{k+1}\langle 
A_0,\dots,A_j, 1, A_{j+1},\dots,A_m\rangle_{m+1,s,r}ds.
\end{align*}
\end{lemma}

\begin{proof}
The only thing that needs justification is the trace norm derivative formula
$$
\frac{d}{ds}\langle A_0,\dots,A_m\rangle_{m,s,r}=
2s\sum_{k=0}^m\langle A_0,\dots,A_k,1,A_{k+1},\dots,A_m\rangle_{m+1,s,r}
$$
for suitable $m,s,r$ and weak pseudodifferential operators $A_j$. So start with the difference
quotient leading to one of the terms on the right hand side.
\begin{align*}
&\frac{1}{2\pi i}\int_l \lambda^{-p/2-r}A_0R\cdots RA_k
\left(\frac{R_{s+\epsilon}-R_s}{\epsilon}\right)A_{k+1}R\cdots RA_mRd\lambda\\
&=(2s+\epsilon)\frac{1}{2\pi i}\int_l \lambda^{-p/2-r}A_0R\cdots RA_k
R_{s+\epsilon}R_sA_{k+1}R\cdots RA_mRd\lambda.
\end{align*}
Now repeat the trick of Lemma \ref{lemma_Bs_trace}, giving
\begin{align*}
&=(2s+\epsilon)\frac{1}{2\pi i(p/2+r-1)(p/2+r-2)\cdots(p/2+r-(2M-1-m))}\times\\
&\times\int_l \lambda^{-p/2-r+(2M-1-m)}
\sum_{j=0}^{2M-1-m}\binom{2M-1-m}{j}\\
&\times \frac{d^j}{d\lambda^j}
\left(a_0Rda_1R^{-1}R^2da_2R^{-2}\cdots 
R^kA_kR^{-k}R_{s+\epsilon}R^{k+1}A_{k+1}R^{-k-1}\cdots  R^mda_mR^{-m}\right)\times\\
&\times\frac{d^{2M-1-m-j}}{d\lambda^{2M-1-m-j}}(R^{m+1})d\lambda.
\end{align*}
Performing the derivatives yields a formula similar to that in the proof of Lemma \ref{lemma_Bs_trace}, 
but in place of the uniformly bounded $B_j$'s, we have uniformly
bounded operators {\em and} one  extra resolvent. Thus
the same trace norm estimates apply and we see that the difference quotients converge in 
trace norm. Thus $\langle  A_0,\dots,A_m\rangle_{m,s,r}$
is trace norm differentiable in $s$, and the derivative goes to zero as $\lambda\to a\pm i\infty$. 
The proof is completed by applying the fundamental theorem of calculus to 
\begin{equation*}
\frac{d}{ds}\left(s^k\langle  A_0,\dots,A_m\rangle_{m,s,r}\right).   \qedhere
\end{equation*}
\end{proof}

A completely analogous argument proves the $\lambda$-trick with our weak smoothness hypotheses.

\begin{lemma}[$\lambda$-trick] 
Let $(\A,\H,\D)$ be a weakly $QC^\infty$ odd modular
spectral triple relative to $(\cn,\phi)$ of dimension $p\geq 1$.
For any integer $m\geq0$, 
operators  $A_j\in \wop^{k_j}$, $j=0,\dots,m$, and $r$ such that $2\Re(r) > 2\sum 
k_j-2m$, $r\not\in \N-p/2$, we have
\begin{equation*} 
-(p/2+r)\langle A_0,\dots,A_m\rangle_{m,s,r+1}=\sum_{k=0}^m\langle 
A_0,\dots,A_k, 1, A_{k+1},\dots,A_m\rangle_{m+1,s,r}.
\end{equation*}
\end{lemma}

\begin{prop}
\label{hii}
Let $(\A,\H,\D)$ be a weakly $QC^\infty$ odd modular
spectral triple relative to $(\cn,\phi)$ of dimension $p\geq 1$.
Let $m=1,3,\dots,2N-1$. Let $\A\otimes\A^{\otimes m}$ have the projective tensor
product topology coming from the seminorms $a\mapsto \Vert WR^k\circ WL^l(a)\Vert_\infty +\Vert WR^k\circ WL^l([\D,a])\Vert_\infty$ on $\A$, and restrict this topology
to the subspace $(\A\otimes\A^{\otimes m})^{\sigma^\phi}$ of $\sigma^\phi$ invariant tensors.
(This can be called the weak $QC^\infty$-topology).
Then the maps
 $$
(\A\otimes\A^{\otimes m})^{\sigma^\phi}\ni a_0\otimes\dots\otimes a_m\mapsto\big[r\mapsto\Phi_{m}^r(a_0,\dots,a_m)\big]\,,
 $$
are continuous multilinear maps from $ (\A\otimes\A^{\otimes m})^{\sigma^\phi}$ to the
functions holomorphic in  $\{z\in\C:\ \Re(z)>(1-m)/2,\,z\not\in\N-p/2\}$, with the 
topology of uniform convergence on compacta.
\end{prop}

\begin{proof}
Let us first fix  $r\in \C$ with $\Re(r)>(1-m)/2$, and set $M=2N-1$. 
Lemma \ref{lemma_Bs_trace} ensures that our functionals are
finite for these values of $r$, and it is an exercise (see \cite[Lemma 7.4]{CPRS2}) to show that these functionals are
holomorphic there. Thus all that we need to do 
is to improve the estimates 
to prove continuity. We do this, following \cite[Proposition 5.18]{CPRS4}, using the $s$- and $\lambda$-tricks. 
We recall that we have defined $M=2N-1$. 
By applying successively the $s$- and $\lambda$-tricks (which commute) $(M-m)/2$ times each, 
we obtain 
\begin{align}
\Phi_{m}^r(a_0,\dots,a_m)&=
2^{(M-m)/2}(M-n)!\prod_{l_1=1}^{(M-m)/2}\frac1{p/2+r-l_1}\prod_{l_2=1}^{(M-m)/2}\frac1{m+l_2}\nonumber\\
&\hspace{2cm}\times\sum_{|k|=M-m}
\int_0^\infty s^M
\langle a_0,1^{k_0},da_1,1^{k_1},\dots,da_m,1^{k_m}\rangle_{M,r-(M-m)/2,s}ds\,,
\label{haa}
\end{align}
where $1^{k_i}=1,1,\dots,1$ with $k_i$ entries. Since $M\leq p+1$, the poles 
associated to the prefactors are outside the region $\{z\in\mathbb C:\Re(z)>(1-m)/2\}$. 
Ignoring the prefactors, setting $n_i=k_i+1$ and $R:=R_{s,t}(\lambda)$, we need to deal 
with the integrals
$$
\int_0^\infty s^M\phi\Big( \gamma\int_l 
\lambda^{-p/2-r-(M-m)/2}a_0R^{n_0}da_1R^{n_1}\cdots da_mR^{n_m}d\lambda\Big)ds\,,
\qquad |n|=M+1\,,
$$
where $l$ is the vertical line $l=\{a+iv:v\in\R\}$ with $a=1/2$.

To estimate the trace norm (using the trace given by restricting $\phi$ to the invariant subalgebra $\cN^{\sigma^\phi}$) we first write
\begin{align*}
&a_0R^{n_0}da_1R^{n_1}\cdots da_mR^{n_m}\\
&\quad=a_0(R^{n_0}da_1 R^{-n_0})(R^{n_0+n_1}da_2R^{-(n_0+n_1)})\cdots (R^{n_0+\cdots+n_{m-1}}da_mR^{-(n_0+\cdots+n_{m-1})}) R^{n_0+\cdots+n_m}.
\end{align*}
Then, using \cite[Lemma 5.2]{CPRS2}, and the fact that $ |n|=M+1$, for each $\epsilon>0$ we obtain $C_\epsilon>0$ such that
\begin{align*}
&\|a_0R^{n_0}da_1R^{n_1}\cdots da_mR^{n_m}\|_1\\
&\leq
\| a_0(R^{n_0}da_1 R^{-n_0})(R^{n_0+n_1}da_2R^{-(n_0+n_1)})\cdots (R^{n_0+\cdots+n_{m-1}}da_mR^{-(n_0+\cdots+n_{m-1})})\|_\infty \, \|  R^{M+1}\|_1   \\
&\qquad\leq \|a_0R^{n_0}da_1R^{n_1}\cdots da_mR^{n_m}R^{-(M+1)}\|_\infty \,C_\epsilon\,((s^2+a^2)+v^2)^{-(M+1)/2+(p+\epsilon)/4}.
\end{align*}

%

The operator norm of the product yields a constant $C(a_0,a_1,\dots,a_m)$ 
depending on $a_0,a_1,\dots,a_m$, which varies continuously as the $a_j$ vary in a weak $QC^\infty$ continuous way. 
Integrating now shows that 
$$
|\Phi^r_m(a_0,\dots,a_m)|\leq |f(r)|\,C_{\epsilon,M,m}C(a_0,a_1,\dots,a_m),
$$
for a function $f$ continuous for $\Re(r)>(1-m)/2$, $r\not\in \N-p/2$ (coming from the prefactor and the integral)
and some constant $C_{\epsilon,M,m}$. 
\end{proof}

\begin{prop}
\label{cocycle}  
Let $(\A,\H,\D)$ be a weakly $QC^\infty$ odd modular
spectral triple relative to $(\cn,\phi)$ of dimension $p\geq 1$, and let $N=[p/2]+1$.
The collection of functionals $\Phi^r=\{\Phi_m^r\}_{m=1}^{2N-1}$, 
$m$ odd, is such that 
\begin{equation} (B^\sigma\Phi^r_{m+2}+b^\sigma\Phi^r_m)(a_0,\dots,a_{m+1})=0\ \ \ m=1,3,\dots,2N-3,\ \ \ 
(B^\sigma\Phi^r_1)(a_0)=0
\label{coc}
\end{equation}
where the $a_i\in\A$, $\sigma=\sigma^\phi_i$ and $b^\sigma, B^\sigma$ are the twisted coboundary operators of cyclic cohomology. 
Moreover, there is a $\delta'$, $0<\delta'<1$
such that  
$b^\sigma\Phi_{2N-1}^r(a_0,\dots,a_{2N})$ is a holomorphic function of $r$ for 
$\Re(r)>-p/2+\delta'/2$.
\end{prop}

\begin{proof}
The proof is just as in \cite[Proposition 7.10]{CPRS2}, using the formula for the twisted coboundaries $b^\sigma,\,B^\sigma$, and the twisted tracial property of $\phi$, 
until we compute
\begin{equation}
(b^\sigma\Phi^r_{2N-1})(a_0,\dots,a_{2N})=\int_0^\infty s^m\langle a_0,[\D,a_1],\dots,[\D^2,a_j],\dots,[\D,a_{m+1}]
\rangle_{2N,r,s}ds.
\label{eq:bdry}
\end{equation}
Since $\A\subset {\rm OP}^0$ we have $[\D^2,\A]\subset {\rm OP}^1$, and then the 
proof is just as in \cite[Proposition 7.10]{CPRS2}. 
\end{proof}


We now specialise to the semifinite case so that we may relate the resolvent cocycle to the index problem 
(that is, to compute spectral flow). 
Proposition \ref{cocycle} establishes that the resolvent cocycle is almost a cocycle, so we 
have the following theorem, proven just as in \cite{CPRS2}.

\begin{thm}
\label{thm:weak-index}
Let $(\A,\H,\D)$ be a weakly $QC^\infty$ odd semifinite
spectral triple relative to $(\cn,\tau)$ of dimension $p\geq 1$. Let $N = [p/2] + 1$
be the least positive integer strictly greater than $p/2$ and let $u \in \cA$ be unitary.
Then
\enveqn{
sf_\tau(\D,u^*\D u) = \frac{1}{\sqrt{2\pi i}} \mathrm{Res}_{r = (1-p)/2} \left( \sum_{m=1,odd}^{2N-1}\Phi^r_m(Ch_m(u)) \right),
}
where $Ch_m(u)$ is defined to be 
$$
Ch_m(u)=(-1)^{(m-1)/2}\,((m-1)/2)!\, u^*\otimes u\otimes\cdots\otimes u\qquad (m+1)\ \mbox{entries}.
$$
\end{thm}
\begin{proof}
This `resolvent index formula' is proved as in \cite{CPRS2}, where the differences for the weak $QC^{\infty}$
assumption are detailed above. 
\end{proof}

{\bf Remark.} In the even case we have a similar statement with $N=[(p+1)/2]$ and the sum runs over even
integers from $m=0$ to $2N$; see \cite{CPRS3} for the $QC^\infty$ case and \cite{S} for the weakly $QC^\infty$ case.

\subsection{The resolvent index formula for modular spectral triples}
\label{subsec:local-index}

Let $(\A,\H,\D)$ be a modular spectral triple relative to $(\cN,\phi)$ with modular group $\sigma^\phi$, 
of spectral dimension $p\geq 1$, and weakly $QC^\infty$ so that
$$
\A\subset {\rm OP}^0,\qquad [\D,\A]\subset \wop^0.
$$
Let $u\in M_n(\A)$ be unitary, $V:\T\to M_n(\C)$ a representation
and suppose that $u$ is $\sigma^\phi\otimes Ad\,V$ invariant.

Lemma \ref{lem:bob-the-builder} constructs a semifinite spectral triple from $(\A,\H,\D)$ and $u$.
The semifinite resolvent index formula, Theorem \ref{thm:weak-index}, then shows that the
resolvent cocycle  defined using the trace $\phi\otimes G$ is `almost' a $b,\,B$ cocycle, and computes the spectral flow 
from $\D$ to $u\D u^*$.

With $N=[p/2]+1$, we have
\begin{equation}
sf_{\phi\otimes G}(\D\ox {\rm Id}_n, 
u(\D\ox {\rm Id}_n) u^*)
=\frac{1}{\sqrt{2\pi i}}\res_{r=(1-p)/2}
\sum_{m=1,odd}^{2N-1}(\Phi_G)^r_m(Ch_m(u)),
\label{eq:sf}
\end{equation}
where $(\Phi_G)^r_m$ is the resolvent cocycle defined using the trace $\phi\otimes G$.
In particular the sum on the right hand side of \eqref{eq:sf} analytically continues 
to a deleted neighbourhood of $r=(1-p)/2$ with {\em at worst} a simple pole 
at $r=(1-p)/2$.

We will compute the $G$ part of the trace, leaving us with a functional defined in terms of
$\phi$.

The Chern character of $u$ is defined to be the (infinite) sum 
$\oplus_j Ch_{2j+1}(u)\in HE_{2j+1}(M_N(\A))$, the entire cyclic homology, with
$$
Ch_{2j+1}(u)=(-1)^j\,j!\, u^*\otimes u\otimes\cdots\otimes u^*\otimes  u\qquad (2j+2)\ \mbox{entries}.
$$
Now in \cite[Lemma 4.1]{W}, Wagner has shown, in a slightly different context, that the map 
$$
G_*:\oplus_jHE_{2j+1}(M_N(\A)^{\s\otimes Ad V})\to \oplus_jHE_{2j+1}^\sigma(\A)
$$
to $\s$-twisted cyclic homology given on chains by
$$
G_*(T_0\otimes\cdots\otimes T_{2j+1})
=\sum_{i_0,i_1,\dots,i_{2j+2}}(V_{-i})_{i_{2j+2},i_0}(T_0)_{i_0,i_1}\otimes (T_1)_{i_1,i_2}\otimes\cdots
\otimes (T_{2j+1})_{i_{2j+1},i_{2j+2}}
$$
is an isomorphism. Now each equivariant unitary with class $[u]\in K_1^\T(A)$ is equivariant for
its own representation of the circle. So it makes sense to regard the representation $V$
as part of the data, so $[u]=[u,V]$. We define $Ch_{2j+1}([u,V])\in HE_{2j+1}^\sigma(\A)$ by
$$
Ch_{2j+1}([u,V])=(-1)^j\,j!\,\sum_I (V_{-i})_{i_{2j+2},i_0}(u^*)_{i_0,i_1}\otimes (u)_{i_1,i_2}\otimes\cdots
\otimes (u_{2j+1})_{i_{2j+1},i_{2j+2}}.
$$
Then it is straightforward to check that this does indeed define an entire twisted cyclic cycle.
Moreover it is immediate from the definitions that
\begin{align*}
sf_{\phi\otimes G}(\D\ox {\rm Id}_n,u(\D\ox {\rm Id}_n) u^*)
&=\frac{1}{\sqrt{2\pi i}}\res_{r=(1-p)/2}
\sum_{m=1,odd}^{2N-1}(\Phi_G)^r_m(Ch_m(u))\\
&=\frac{1}{\sqrt{2\pi i}}\res_{r=(1-p)/2}
\sum_{m=1,odd}^{2N-1}\Phi^r_m(Ch_m([u,V]))
\end{align*}
Here $\Phi^r_m$ is the resolvent cocycle given by the modular spectral triple.

We now collect the results proved above into a statement describing the resolvent index formula for weakly smooth modular spectral triples.

\begin{thm}
\label{thm_lif_odd}
For a weakly $QC^\infty$ odd 
modular spectral triple $(\A,\H,\D)$ relative to $(\cn,\phi)$ of spectral dimension $p\geq 1$,
and with $N=[p/2]+1$, the  function valued cochain 
$(\Phi^r_m)_{m=1,\dots,2N-1}$ is a twisted cyclic cocycle modulo
cochains with values in functions holomorphic in a half-plane containing $(1-p)/2$.
Moreover, for $[u,V]\in K_1^\T(\A)$ with representative $u\in M_n(\A)$ we have
\enveqn{
sf_{\phi\otimes G}(\D\ox {\rm Id}_n,u^*(\D\ox {\rm Id}_n) u) = \frac{1}{\sqrt{2\pi i}} \mathrm{Res}_{r = (1-p)/2} \left( \sum_{m=1,odd}^{2N-1}\Phi^r_m(Ch_m([u,V])) \right).
}
In particular, there is a well-defined map
$$
K_1^\T(\A)\mapsto \R,\qquad [u,V]\mapsto sf_{\phi\otimes G}(\D\ox {\rm Id}_n,u^*(\D\ox {\rm Id}_n) u).
$$
\end{thm}


Though we have not proved it here, a similar result is true in the even case, see \cite{S}.

\begin{thm}
\label{thm_lif_even}
For a weakly $QC^\infty$ even modular spectral triple $(\A,\H,\D,\gamma)$ relative to $(\cn,\phi)$
of spectral dimension $p\geq 1$,
and with $M=[(p+1)/2]$, the  function valued cochain 
$(\Phi^r_m)_{m=0,\dots,2M}$ is a twisted cyclic cocycle modulo
cochains with values in functions holomorphic in a half-plane containing $(1-p)/2$.
Moreover, for $[P,V]\in K_0^\T(\A)$ with representative $P\in M_n(\A)$ and $\D_+=\frac{1}{4}(1-\gamma)\D(1+\gamma)$ we have
\enveqn{
{\rm Index}_{\phi\ox G}(P(\D_+\ox{\rm Id}_n)P)= \frac{1}{\sqrt{2\pi i}} \mathrm{Res}_{r = (1-p)/2} \left( \sum_{m=0,even}^{2M}\Phi^r_m(Ch_m([P,V])) \right).
}
In particular, there is a well-defined map
$$
K_0^\T(\A)\mapsto \R,\qquad [P,V]\mapsto {\rm Index}_{\phi\ox G}(P(\D_+\ox{\rm Id}_n)P).
$$

\end{thm}

{\bf Remark.} The Chern character of an equivariant projection is
\begin{equation}
Ch_0([P,V])={\rm Tr}(V_{-i}P),\quad Ch_{2k}([P,V])=(-1)^k\frac{(2k)!}{k!}\sum (V_{-i}(P-1/2))_{i_0i_1}\otimes
P_{i_1i_2}\otimes\cdots\otimes P_{i_{2k}i_0}.
\label{eq:even-chern}
\end{equation}

Finally, the next two results relate the even index given by the resolvent index formula above 
back to the $K$-theory valued index pairing between the $KK$-class defined by the modular spectral triple and equivariant $K$-theory.

\begin{lemma}
Let $(\A,\H,\D)$ be a modular spectral triple  relative to $(\cn,\phi)$. Let $J_\phi\subset \cN$ be the ideal from Definition \ref{defn_j_phi}, and $J_\phi^\sim$
its unitisation.
Let $E \in M_{k}(J_{\phi}^\sim)$ be a $\sigma^\phi \ox Ad \, W$-invariant projection, for the associated 
representation $W \colon \bT \to M_k(\bC)$, so that $[E, W] \in K_{0}^{\bT}(J_{\phi}^\sim)$. Define 
\enveqn{
\phi_{\ast}([E, W]) := (\phi \ox G_{W})(E) \in [0, \infty],
}
where $G_{W}(T) = \Tr(W_{-i} T)$, for $T \in M_{k}(\bC)$. Then $\phi_{\ast}$ is a well-defined map on the semigroup of Murray-von Neumann
equivalence classes of equivariant projections in $J_\phi^\sim\otimes \K$, where $\K$ is the compact operators. The Grothendieck group of the sub-semigroup 
for which $\phi_*$ takes finite values is (isomorphic to) a subgroup of
$K_{0}^{\bT}(J_{\phi})$, and we call this the domain of $\phi_*$.
\end{lemma}
\begin{proof}
Let $W_{1} \colon \bT \to M_{n}(\bC)$ and $W_{2} \colon \bT \to M_{m}(\bC)$ be representations. Let $E_{1} \in M_{n}(J_{\phi})$ 
denote a $\sigma \ox Ad \, W_{1}$ projection, and let $E_{2} \in M_{m}(J_{\phi})$ denote a $\sigma^\phi \ox Ad \, W_{2}$ projection. 
Suppose that $[E_{1}, W_{1}]$ and $[E_{2}, W_{2}]$ are equivariantly Murray-von Neumann equivalent (\cite[Definition 3.1]{W}), 
meaning there exists some $S \in M_{m \times n}(J_{\phi})$ such that 
\enveqn{
S^{\ast} S = E_{1}, \qquad S S^{\ast} = E_{2}, \qquad \text{and} \qquad W_{2, z} S = S W_{1, z} \quad \text{for all} \ \ z \in \bT.
}
Then we compute
\envaln{
\phi_{\ast}([E_{1}, W_{1}]) &= (\phi \ox G_{W_{1}})(E_{1}) = \phi \left( \Tr_{n}(W_{1, -i} E_{1}) \right) = \phi \left( \Tr_{n}(W_{1, -i} S^{\ast} S) \right) \\
&= \phi \left( \Tr_{n}(S W_{1, -i} S^{\ast}) \right).
}
Now, by analytically continuing, $S W_{1, -i} = W_{2, -i} S$, so
\envaln{
\phi_{\ast}([E_{1}, W_{1}]) &= \phi \left( \Tr_{m}(W_{2, -i} S S^{\ast}) \right) = \phi \left( \Tr_{m}(W_{2, -i} E_{2}) \right) = \phi_{\ast}([E_{2}, W_{2}]).
}
Using the universal property of the Grothendieck group, we see that the Grothendieck group of equivalence classes
for which $\phi_*$ takes finite values may be regarded as a subgroup of $K_{0}^{\bT}(J_{\phi})$. On this subgroup, $\phi_*$ is
well-defined.
\end{proof}

\begin{thm}
\label{thm_ch3_kk_index_resolvent}
Let $(\A,\H,\D,\gamma)$ be a weakly $QC^\infty$ even modular spectral triple  relative to $(\cn,\phi)$
of spectral dimension $p\geq 1$, and $[P,V]\in K_0^\bT(A)$. Let $B_\phi\subset J_\phi$ be as in Definition \ref{def:bee-phi}, and let $i:B_\phi\to J_\phi$
be the inclusion.
Then
$i_*([P, V] \ox_{A} [(B_\phi,F_{\cD})]) \in K_{0}^{\bT}(J_{\phi})$ is in the domain of $\phi_{\ast}$. Furthermore,
\enveqn{
\phi_{\ast}(i_*([P, V] \ox_{A} [(B_\phi,F_{\cD})] )) = \res_{r = (1-p)/2} \left( \sum_{m=0,even}^{2N}\Phi^{r}_{m}(Ch_{m}([P, V])) \right).
}
\end{thm}
\begin{proof}
Given the modular spectral triple $(\cA,\cH,\cD,\gamma,\cN,\phi)$, 
we define  $[(B_\phi,F_{\cD})] \in KK^{0, \bT}(A, B_{\phi})$. Also, let $V \colon \bT \to M_n(\bC)$ be a representation and 
$P \in M_n(\cA)$ a projection which is $\sigma^\phi \otimes Ad \,V$ invariant, so that we obtain a class $[P, V] \in K_{0}^{\bT}(A)$.

Define the projections
\enveqn{
N_{\pm} := \ker(P (\cD \ox \mathrm{Id}_{n})^{\pm} P),
}
so that
\enveqn{
\Ind_{\phi \ox G}(P (\cD \ox \mathrm{Id}_{n})^{+} P) = (\phi \ox G)(N_{+}) - (\phi \ox G)(N_{-}).
}
By the construction of the semifinite spectral triple $(C^\infty(P), \cH \ox \bC^{n}, \cD \ox \mathrm{Id}_{n}, \cM_{n}, \phi \ox G)$, we have 
$N_{\pm} \in \cK((M_{n}(\cN))^{\sigma^\phi \ox Ad \, V}, \phi \ox G)$,
since the $N_{\pm}$ are kernel projections, and
\enveq{
\label{eqn_ch3_N_leq_pdp}
N_{\pm} \leq (P + (P (\cD \ox \mathrm{Id}_{n}) P)^{2})^{-1}.
}
Also, the $\sigma^\phi \ox Ad \, V$-invariance of $P$ implies the same invariance for $N_{\pm}$. 

We now want to show that we also have
$N_{\pm} \in M_{n}(B_{\phi})$, so that they define classes in $K_{0}^{\bT}(B_{\phi})$. We do this by proving that the 
operator $(P + (P (\cD \ox \mathrm{Id}_{n}) P)^{2})^{-1} \in M_{n}(B_{\phi})$, then applying Equation \eqref{eqn_ch3_N_leq_pdp} again
to see that $N_{\pm} \in M_{n}(B_{\phi})$.


For brevity let $\cD_{n} := \mathrm{Id}_{n} \ox \cD$. Consider the operator 
$(P + (P\cD_{n}P)^{2})^{-1} \colon P(\bC^{n} \ox B_\phi) \rightarrow P(\bC^{n} \ox B_\phi)$. 
The inverse exists because $P$ acts as the identity on $P(\bC^{n} \ox B_\phi)$, and $(P \cD_{n} P)^{2} \geq 0$. 
The adjointable endomorphisms on $P(\bC^{n} \ox B_\phi)$ are $P M_{n}(M(B_\phi)) P$, where $M(B_\phi)$ is the multiplier algebra,
while the compact operators are $P M_{n}(B_\phi) P$. 
A priori, we know only that $(P + (P\cD_{n}P)^{2})^{-1}$ is bounded on $P(\bC^{n} \ox B_\phi)$.

To show the compactness of $(P + (P\cD_{n}P)^{2})^{-1}$, we compute
\envaln{
(P + (P\cD_{n}P)^{2})^{-1} &= (P + P[\cD_{n}, P] \cD_{n} P + P\cD_{n}^{2}P)^{-1} \\
&= (P + P\cD_{n}^{2}P)^{-1} + \Big[ (P + P[\cD_{n}, P] [\cD_{n}, P] P + P\cD_{n}^{2}P)^{-1} - (P + P\cD_{n}^{2}P)^{-1} \Big],
}
where the last line follows from the observation
\enveqn{
P[\cD_{n}, P]P = P[\cD_{n}, P^{2}]P = P(P[\cD_{n}, P] + [\cD_{n}, P]P)P = 2P[\cD_{n}, P]P = 0,
}
so that $P[\cD_{n}, P] [\cD_{n}, P] P = P[\cD_{n}, P] \cD_{n} P$. Now, the algebraic result $\alpha^{-1} - \beta^{-1} = \beta^{-1} (\beta - \alpha) \alpha^{-1}$ yields
\envaln{
&(P + P[\cD_{n}, P] [\cD_{n}, P] P + P\cD_{n}^{2}P)^{-1} - (P + P\cD_{n}^{2}P)^{-1} \\
& \qquad \qquad \qquad = - (P + P\cD_{n}^{2}P)^{-1} \Big( P[\cD_{n}, P] [\cD_{n}, P] P \Big) (P + P[\cD_{n}, P] [\cD_{n}, P] P + P\cD_{n}^{2}P)^{-1}.
}
Hence
\enveqn{
(P + (P\cD_{n}P)^{2})^{-1} = (P + P\cD_{n}^{2}P)^{-1} B(P),
}
where $B(P)$ is a bounded operator, given by
\enveqn{
B(P) = 1 - (P[\cD_{n}, P] [\cD_{n}, P] P) (P + P[\cD_{n}, P] [\cD_{n}, P] P + P\cD_{n}^{2}P)^{-1}.
}

Now consider $(1 + \cD_{n})^{-1} \colon \bC^{n} \ox B_\phi \rightarrow \bC^{n} \ox B_\phi$. Then we have
\enveq{
\label{eqn_ch3_parametrix}
(P + P\cD_{n}^{2}P) \, P (1 + \cD_{n})^{-1} P = P + P [\cD_{n}^{2}, P] (1 + \cD_{n}^{2})^{-1} P=P+PC(P)(1+\D_n^2)^{-1/2}P.
}
Here $C(P)$ is bounded since $P\in M_n(\A)\subset OP^0$ (where $OP^0$ is defined using $\D_n$). 

Now $(1 + \cD^{2})^{-1/2} \in  B_\phi$ by definition, 
so $(1 + \cD_{n}^{2})^{-1/2} \in M_{n}(B_\phi)$. Hence
\enveqn{
P [\cD_{n}^{2}, P] (1 + \cD_{n}^{2})^{-1} P \in P M_{n}(B_\phi) P,
}
so Equation \eqref{eqn_ch3_parametrix} now implies that
\enveqn{
(P + P\cD_{n}^{2}P)^{-1} \in P M_{n}(B_\phi) P.
}

We know $B_\phi$ is an ideal in the endomorphisms, so Equation \eqref{eqn_ch3_N_leq_pdp} 
now implies that $N_{\pm} \in M_{n}(B_\phi)$. By the $\sigma^\phi \ox Ad \, V$-invariance of $N_{\pm}$, 
we have $[N_{\pm}, V] \in K_{0}^{\bT}(B_\phi)$. Then
\enveq{
\label{eqn_ch3_phig_phiast}
(\phi \ox G)(N_{+}) - (\phi \ox G)(N_{-}) = \phi_{\ast}\left(i_*\left( [N_{+}, V] - [N_{-}, V] \right)\right).
}

In order to compare Equation \eqref{eqn_ch3_phig_phiast} to the Kasparov product 
$[P, V] \ox_{A} [(B_\phi, F_{\cD})]$, we rewrite the classes $[N_{\pm}, V]$ as Kasparov modules. We have
\enveqn{
[N_{+}, V] - [N_{-}, V] = \left[ \left( N_{+}(\bC^{n} \ox B_\phi) \oplus N_{-}(\bC^{n} \ox B_\phi), \, 0, \, \bma N_{+} & 0 \\ 0 & -N_{-} \ema, \, V \oplus V \right) \right],
}
where $N_{+}(\bC^{n} \ox B_\phi) \oplus N_{-}(\bC^{n} \ox B_\phi)$ 
is the right Hilbert $B_\phi$-module, $0$ is the operator, $\bma N_{+} & 0 \\ 0 & -N_{-} \ema$ is the 
grading and $V \oplus V$ is the $\bT$-action giving the equivariance.

Now the operator $P (\mathrm{Id}_{n} \ox F_{\cD})^{+} P$ gives an isomorphism from 
$(1 - N_{+})(\bC^{n} \ox B_\phi)$ to $(1 - N_{-})(\bC^{n} \ox B_\phi)$. 
Hence, the Kasparov module constructed from $(1 - N_{\pm})(\bC^{n} \ox B_\phi)$ and $P (\mathrm{Id}_{n} \ox F_{\cD})^{+} P$ has trivial class. Consequently,
\envaln{
& \left[ \left( N_{+}(\bC^{n} \ox B_\phi) \oplus N_{-}(\bC^{n} \ox B_\phi), \, 0, \, \bma N_{+} & 0 \\ 0 & -N_{-} \ema, \, V \oplus V \right) \right] \\
& \qquad \qquad \qquad \qquad \qquad \qquad \qquad \qquad 
= \left[ \left( \bC^{n} \ox B_\phi, \, P (\mathrm{Id}_{n} \ox F_{\cD}) P, \, \mathrm{Id}_{n} \ox \gamma, \ V \right) \right].
}

Finally, observe that, see \cite{B} for example, we have an explicit representative of the Kasparov product
\enveqn{
[P, V] \ox_{A} [( B_\phi,F_{\cD})] = \left[ \left( \bC^{n} \ox B_\phi, \, P (\mathrm{Id}_{n} \ox F_{\cD}) P, \, \mathrm{Id}_{n} \ox \gamma, \ V \right) \right].
}


Reiterating the above results, we have proved that
\envaln{
\Ind_{\phi \ox G}(P (\cD \ox \mathrm{Id}_{n})^{+} P) &= (\phi \ox G)(N_{+}) - (\phi \ox G)(N_{-}) \\
&= \phi_{\ast}\left(i_*\left( [N_{+}, V] - [N_{-}, V] \right)\right) \\
&= \phi_{\ast}\left(i_*\left( \left[ \left( \bC^{n} \ox B_\phi, \, P (\mathrm{Id}_{n} \ox F_{\cD}) P, \, \mathrm{Id}_{n} \ox \Gamma, \ V \right) \right] \right)\right) \\
&= \phi_{\ast}\left(i_*\left( [P, V] \ox_{A} [(B_\phi,F_{\cD})] \right)\right). \qedhere
}
\end{proof}

\section{The local index formula for the Podle\'s sphere}
\label{sec:pods-chern}

In this section we will explicitly compute a twisted $b,B$ cocycle for the 
(modular) spectral triple over the Podle\'s sphere first investigated in \cite{DS}. We
do this by applying the modified pseudodifferential calculus of \cite{NT} to the twisted resolvent
cocycle of the previous section. Having done this, we construct some equivariant projections
for a circle action arising from the Haar state and compute the index pairing via a residue formula, yielding a local index formula. 

\subsection{The modular spectral triple for the Podle\'s sphere}
\label{subsec:pods-mod}

We first recall (see \cite{KS}) that the quantum algebra $\cA = \cO(SU_{q}(2))$, for $q \in [0, 1]$, is 
generated by elements $a,\, b,\, c,\, d$ modulo the relations

\begin{align*}
ab = qba, \ \ \ ac = qca, \ \ \ bd &= qdb, \ \ \ cd = qdc, \ \ \ bc = cb \\
ad = 1 + qbc, \ & \ \ da = 1 + q^{-1}bc \\
a^{\ast} = d, \ \ \ b^{\ast} = -qc, \ &\ \ c^{\ast} = -q^{-1}b, \ \ \ d^{\ast} = a.
\end{align*}

The Podle\'s sphere, which we denote by $\cB$, is (isomorphic to)
the unital $\ast$-subalgebra of $\cO(SU_{q}(2))$ 
generated by $q^{-1}ab$, $-cd$ and $-q^{-1}bc$.

%


Recall that for each $l \in \half \bN_0$, there is a 
unique (up to unitary equivalence) irreducible corepresentation $V_l$ of 
the coalgebra $\A$ of dimension $2l+1$, and that $\A$ is
cosemisimple. That is, if we fix a vector space basis in
each of the $V_l$ and denote by $t^l_{r,s} \in \A$ the corresponding
matrix coefficients, then we have the following analogue of the
Peter-Weyl theorem.


\begin{thm}[{\cite[Theorem 4.13]{KS}}] \label{thm_rep_basis}
Let 
$ I_{l} := \{ -l, -l+1, \ldots, l-1,l \} $.
Then the set $\{ \re{l}{r}{s} \ | \ l \in
 \half \bN_{0}, \ r, s \in I_{l} \}$ is a vector space basis of $\A$.
\end{thm}

This will be referred to as the Peter-Weyl basis. With a suitable choice of basis in $V_{\half}$, one has

\envaln{
a &= \re{\half}{-\half}{-\half}, & b &= \re{\half}{-\half}{\half}, & c &= \re{\half}{\half}{-\half}, 
& d &= \re{\half}{\half}{\half}.
}

The expressions for the Peter-Weyl basis elements as linear combinations of the polynomial basis elements can be found in \cite[Section 4.2.4]{KS}.

The algebra $\cA$ has a useful direct sum decomposition. For $m, \, n \in \bZ$ where $m - n$ is even, define

\enveqn{
\cA[m, n] := \mathrm{span} \{ a^{\frac{1}{2}(m+n)}b^{k + \frac{1}{2}(m-n)}c^{k}, \ b^{k + \frac{1}{2}(m-n)}c^{k}d^{-\frac{1}{2}(m+n)} \colon k + \min\{ 0, \, \tfrac{1}{2}(m-n) \} \in \bN_{0} \},
}

and for $m - n$ odd, let $\cA[m, n] := \{ 0 \}$. Then 

\enveqn{
\cA = \bigoplus_{m, n \in \bZ} \cA[m, n], \qquad \text{and} \qquad \cA[m_{1}, n_{1}] \cdot \cA[m_{2}, n_{2}] \subseteq \cA[m_{1}+m_{2}, n_{1}+n_{2}].
}

With this notation, we have $\cB = \bigoplus_{m \in \bZ} \cA[m, 0]$.


Let $h$ be the Haar state on the universal $C^{\ast}$-completion of the $\ast$-algebra $\cA$, whose value 
on the Peter-Weyl basis is $h(\re{l}{r}{s}) = \delta_{l, 0}$. Define an automorphism $\vartheta$ on $\cA$ by

\enveqn{
\vartheta(a) = q^{2} a, \quad \vartheta(b) = b, \quad \vartheta(c) = c, \quad \vartheta(d) = q^{-2} d.
}

Then $\vartheta$ is the modular automorphism for the Haar state, in the sense that 
$h(\alpha \beta) = h(\vartheta(\beta) \alpha)$ for all $\alpha, \beta \in \cA$. For all $n \in \bZ$ define

\enveqn{
\cH_{n} := L^{2} \left( {\rm span} \left\{ \re{l}{r}{\frac{n}{2}} \colon l \in \tfrac{n}{2} + \bN_{0}, r \in I_{l} \right\}, h \right).
}

The left action of the dual Hopf algebra to $\cA$ provides the unbounded operators $\p_e \colon \cH_{n} \rightarrow \cH_{n+2}$ and $\p_f \colon \cH_{n} \rightarrow \cH_{n-2}$ given by

\envaln{
\p_e(\re{l}{r}{s}) &=  \sqrt{\left[ l + \half \right]_{q}^{2} - \left[ s+ \half \right]_{q}^{2} } \,\,\re{l}{r}{s+1}, & \p_f(\re{l}{r}{s}) &= \sqrt{\left[ l + \half \right]_{q}^{2} - \left[ s- \half \right]_{q}^{2} } \,\, \re{l}{r}{s-1}
}

where our definition of  the $q$-number $[a]_{q}$ is

\enveqn{
[a]_{q} := \frac{q^{-a} - q^{a}}{q^{-1} - q}=Q(q^{-a} - q^{a}) \qquad \text{for any} \ a \in \bC,
}

and we abbreviated $Q := (q^{-1} - q)^{-1} \in (0, \infty)$. Finally, we define an unbounded linear operator 
$\Delta_{R}$ on $\cA \subset \bigoplus \cH_{n}$ by

\enveqn{
\Delta_{R}(\re{l}{r}{s}) := q^{2r} \re{l}{r}{s}.
}

\begin{defn}
Define the  Hilbert space 
$\cH := \cH_{1} \oplus \cH_{-1}$, and represent  $\cB$ on $\H$ by left multiplication. 
The Hilbert 
space $\cH$ is graded by $\gamma := \bma 1 & 0 \\ 0 & -1 \ema$. Define the weight 
$\Psi_R$ on $\B(\H)$ by $\Psi_R(T):={\rm Trace}(\Delta_R^{-1/2}T\Delta_R^{-1/2})$.
Finally, on a suitable domain in $\H$, define the self-adjoint operator $\cD := \bma 0 & \p_{e} \\ \p_{f} & 0 \ema$.  
\end{defn}


In fact $(\B,\H,\D,\gamma)$ defines an honest spectral triple, \cite{DS}, 
(i.e. a modular spectral triple with von Neumann algebra $B(\H)$ and 
weight given by the operator trace)  which is $\epsilon$-summable for all $\epsilon>0$.


%
%

\begin{lemma}
\label{lemma_podles_wop_zero}
The data $(\B,\H,\D,\B(\H),\Psi_R)$ defines a weakly $QC^\infty$ even modular spectral triple, which is finitely summable with spectral dimension 2.
\end{lemma}
\begin{proof}
We first show that the data produces a modular spectral triple. Certainly $\B$ is a separable $\ast$-subalgebra of 
$\B(\cH)$, and by construction the modular automorphism group of $\Psi_{R}$ is $\vartheta^{-1}_{t}$, and $\cB$ 
consists of analytic vectors for $\vartheta^{-1}_{t}$.

Also, the commutators $\left[ \cD , \beta \right]$ extend to 
bounded operators for all $\beta \in \cB$,  given by

\begin{equation}
\label{eqn_podles_com}
d\beta := \left[ \cD , \beta \right] 
= \bma 0 & q^{-\frac{1}{2}} \p_{e}(\beta) \\ q^{\frac{1}{2}} \p_{f}(\beta) & 0 \ema.
\end{equation}
We also observe that $\gamma = \gamma^{\ast}$ and $\gamma^{2} = I$, and by construction $\gamma \cD + \cD \gamma = 0$.

Now, for $T \in \B(\cH)$ set $T^{+} = (1 + \gamma) T (1 + \gamma) /4$ and $T^{-} = (1 - \gamma) T (1 - \gamma) /4$. 
From the definition of the operator trace, and using the normalised Peter-Weyl basis $\xi^{l}_{r, j} := \re{l}{r}{j} / \Vert \re{l}{r}{j} \Vert$, we find for $T \geq 0$ that
\enveqn{
\Psi_{R}(T) = \sum_{l, r} q^{-2r} \left( \la \xi^{l}_{r, 1/2}, T^{+} \xi^{l}_{r, 1/2} \ra + \la \xi^{l}_{r, -1/2}, T^{-} \xi^{l}_{r, -1/2} \ra \right).
}
We first observe from the above formula that the finite rank operators are in the domain of $\Psi_{R}$, so 
$\Psi_{R}$ is semifinite. Next, we see that $\Psi_{R}$ is a sum of vector states with orthogonal support, as the Peter-Weyl basis is orthogonal. Hence $\Psi_{R}$ is strictly semifinite.

The Peter-Weyl basis elements can be used to construct a common eigenbasis for $\cD$ and $\Delta_{R}$ on $\cH$, 
so the spectral projections of $\cD$ and $\Delta_{R}$ commute. We conclude that $\cD$ is affiliated to the fixed point 
algebra $\cM := \B(\cH)^{\vartheta^{-1}}$. All that remains to be proved is that $(1 + \cD^{2})^{-1/2} \in \cK(\cM, \Psi_{R}|_{\cM})$. 
To establish this, we observe that $\cD^{2}$ has the following spectral projections
\enveqn{
\cP_{l} \bca \re{k}{r}{1/2} \\ 0 \eca := \delta_{l, k} \bca \re{k}{r}{1/2} \\ 0 \eca, \qquad \qquad \cP_{l} \bca 0 \\ \re{k}{r}{-1/2} \eca := \delta_{l, k} \bca 0 \\ \re{k}{r}{-1/2} \eca,
}
for $l = 1/2, \, 3/2, \, \ldots$, which correspond to the eigenvalues $[l + \half]_{q}^{2}$. Now, 
$\Psi_{R}(\cP_{l}) = \sum_{r = -l}^{l} q^{-2r} = [2l+1]_{q}$, and the sum $\sum_{l = \frac{1}{2}, \frac{3}{2}, \ldots} (1 + [l + \half]_{q}^{2})^{-1/2} < \infty$ implies that
\enveqn{
(1 + \cD^{2})^{-1/2} = \sum_{l = \frac{1}{2}, \frac{3}{2}, \ldots} (1 + [l + \half]_{q}^{2})^{-1/2} \cP_{l}
}
is norm convergent. Hence $(1 + \cD^{2})^{-1/2} \in \cK(\cM, \Psi_{R}|_{\cM})$, and so $(\cB,\cH,\cD,\gamma,\B(\cH),\Psi_{R})$ is a modular spectral triple.
The spectral dimension is shown to be 2 in \cite{KW}.

We now prove that $\cB \subset \mathrm{OP}^{0}, \, [\cD, \cB] \subset \wop^{0}$, so that the modular spectral triple is 
weakly $QC^\infty$. The first statement is proved in \cite[Proposition 3.2]{NT}. To prove the second statement we show 
that for all $\beta \in \cB$ and $z \in \bC$, the operators  $(1+\cD^{2})^{-z} [\cD, \beta] (1+\cD^{2})^{z} \in \B(\cH)$, as per 
Lemma \ref{lem:bdd-conj}.  We begin by observing that $\cD^{2}$ has eigenbasis given by
\envaln{
\cD^{2} \bca t^{l}_{r, \frac{1}{2}} \\ 0 \eca &= [l + \tfrac{1}{2}]^{2} \bca t^{l}_{r, \tfrac{1}{2}} \\ 0 \eca, &
\cD^{2} \bca 0 \\ t^{l}_{r, -\frac{1}{2}} \eca &= [l + \tfrac{1}{2}]^{2} \bca 0 \\ t^{l}_{r, -\tfrac{1}{2}} \eca.
}

Now we consider $\beta \in \B$ to be of the form $t^{p}_{r,0}$ (as finite linear combinations of these elements span $\cB$). 
Then the commutator $[\cD, t^{p}_{r, 0}] = \left(\begin{array}{cc} 0 & \kappa_{1} t^{p}_{r, 1} \\ \kappa_{2} t^{p}_{r, -1} & 0 \end{array}\right)$ 
for some coefficients $\kappa_{1}, \kappa_{2}$. We expand the product $t^{p}_{r, 0} t^{l}_{s, \frac{1}{2}}$ using the Clebsch-Gordan 
coefficients (see \cite{DLSSV,KS}), giving
\enveqn{
(1+\cD^{2})^{-z} t^{p}_{r, 0} (1+\cD^{2})^{z} \bca t^{l}_{s, \frac{1}{2}} \\ 0 \eca 
= (1 + [l + \tfrac{1}{2}]_{q}^{2})^{z} \sum_{k = |l-p|}^{l+p} (1 + [k + \tfrac{1}{2}]_{q}^{2})^{-z} c_{s, r}^{p, l, k} 
\bca t^{k}_{s+r, \frac{1}{2}} \\ 0 \eca
}
where $c_{s, r}^{p, l, k}$ is some product of Clebsch-Gordan coefficients that will be subsumed later.

The norm of $(1+\cD^{2})^{-z} t^{p}_{r, 0} (1+\cD^{2})^{z} \left(\begin{array}{c} t^{l}_{s, \frac{1}{2}} \\ 0 \end{array}\right)$ 
can be computed using the orthogonality of the Peter-Weyl basis, so

\begin{align*}
\left\Vert (1+\cD^{2})^{-z} t^{p}_{r, 0} (1+\cD^{2})^{z} \left(\begin{array}{c} t^{l}_{s, \frac{1}{2}} \\ 0 \end{array}\right) \right\Vert^{2} 
&= \sum_{k = |l-p|}^{l+p} \left( \frac{1 + [l + \tfrac{1}{2}]^{2}}{1 + [k + \tfrac{1}{2}]^{2}} \right)^{2\Re(z)} 
\left| c_{s, r}^{p, l, k} \right|^{2} \left\Vert \left(\begin{array}{c} t^{k}_{s+r, \frac{1}{2}} \\ 0 \end{array}\right) \right\Vert^{2}.
\end{align*}

Let $M_{l, p} := \max_{|l-p| \leq k \leq l+p} \{ \left( (1 + [l + \tfrac{1}{2}]^{2})/(1 + [k + \tfrac{1}{2}]^{2}) \right)^{2\Re(z)} \}$, then

\begin{align*}
\left\Vert (1+\cD^{2})^{-z} t^{p}_{r, 0} (1+\cD^{2})^{z} \left(\begin{array}{c} t^{l}_{s, \frac{1}{2}} \\ 0 \end{array}\right) \right\Vert^{2} &\leq M_{l, p} \sum_{k = |l-p|}^{l+p} \left| c_{s, r}^{p, l, k} \right|^{2} \left\Vert \left(\begin{array}{c} t^{k}_{s+r, \frac{1}{2}} \\ 0 \end{array}\right) \right\Vert^{2} = M_{l, p} \left\Vert t^{p}_{r, 0} \left(\begin{array}{c} t^{l}_{s, \frac{1}{2}} \\ 0 \end{array}\right) \right\Vert^{2}.
\end{align*}

It remains to show that there exist finite $M_{p}$ such that $M_{l, p} \leq M_{p}$ for all $l \geq 0$. 
Let $\varepsilon_{k} = Q(1 - q^{2k})$ so that $[k]_{q} = q^{-k}\varepsilon_{k}$. Then for all $l \geq p + \half$,

\begin{equation*}
\frac{1 + [l + \tfrac{1}{2}]^{2}}{1 + [|l-p| + \tfrac{1}{2}]^{2}} 
= \frac{\varepsilon_{l + \frac{1}{2}}^{2} + q^{2l+1}}{q^{2p}\varepsilon_{l -p + \frac{1}{2}}^{2} + q^{2l+1}}, \quad 
\implies \quad \frac{\varepsilon_{1}^{2}}{1 + q^{2p}Q^{2}} \leq \frac{1 + [l + \tfrac{1}{2}]^{2}}{1 + [|l-p| + \tfrac{1}{2}]^{2}} 
\leq \frac{Q^{2} + 1}{q^{2p} \varepsilon_{1}^{2}}.
\end{equation*}

It follows that the operator $(1+\D^{2})^{-1} t^{p}_{r, 0} (1+\D^{2})$ is bounded on the set of 
vectors of the form $\left(\begin{array}{c} t^{l}_{s, \frac{1}{2}} \\ 0 \end{array}\right)$. 
The same calculation can be performed for the vectors $\left(\begin{array}{c} 0 \\ t^{l}_{s, -\frac{1}{2}} \end{array}\right)$, 
and again for the operators 
$\left(\begin{array}{cc} 0 & t^{p}_{r, 1} \\ 0 & 0 \end{array}\right)$ 
and $\left(\begin{array}{cc} 0 & 0 \\ t^{p}_{r, -1} & 0 \end{array}\right)$, completing the proof.
\end{proof}

\subsection{The residue cocycle for the Podle\'s sphere}
\label{subsec:head-aches}

Lemma \ref{lemma_podles_wop_zero} shows that the 
modular spectral triple $(\cB, \cH, \cD)$ satisfies the 
hypotheses of Theorem \ref{thm_lif_even}. Hence we can employ the 
resolvent cocycle to compute index pairings
with equivariant $K$-theory, or at least those classes which can be represented as projections over $\B$. 
As $\Delta_{R}$ implements the modular automorphism $\vartheta$,
then it follows that the weight $\Psi_{R}$ is $\vartheta^{-1}$-twisted. The resolvent 
cocycle, which we denote by, $( \phi_{m}^{r})_{m=0,2}$ therefore lives is $\vartheta^{-1}$-twisted cohomology.

To simplify the computation of the resolvent cocycle, we 
would like to have a version of the pseudodifferential calculus.
A simple replacement for the pseudodifferential calculus for this example was presented in \cite{NT}. 

\begin{lemma}[{\cite[Corollary 3.3]{NT}}]
\label{lemma_podles_pdc}
Define $\chi := \bma q^{-1} & 0 \\ 0 & q \ema$ on $\mathcal{H}_{1} \oplus \mathcal{H}_{-1}$. 
For any $\beta \in \B$ there exists an analytic function 
$z \mapsto M(z) \in \wop^0\subset \B(\cH)$ with at 
most linear growth on vertical strips such that
\begin{equation*}
|D|^{-z} d\beta = d\beta \chi^{z} |D|^{-z} + M(z) |D|^{-z-1} = \chi^{-z} d\beta |D|^{-z} + M(z) |D|^{-z-1}.
\end{equation*}
\end{lemma}

%

We can now use this pseudodifferential calculus to simplify the computation of the resolvent cocycle, 
$(\phi_{0}^{r},\, \phi_{2}^{r})$, and arrive
at a twisted version of the local index formula.

The first simplification we make is to discard the $1$ from the resolvent, replacing 
$R_{s}(\lambda) = (\lambda - (1 + s^{2} + \cD^{2}))^{-1}$ with 
$R_{s}(\lambda) = (\lambda - (s^{2} + \cD^{2}))^{-1}$. 
This is possible because $\cD$ is invertible in this example, so we can employ the method of 
\cite[Section 5.3]{CPRS4}, in particular {\cite[Proposition 5.20]{CPRS4}}. (The transgression cochain
defined there is well defined for weakly $QC^\infty$ modular spectral triples since $\D\in {\rm OP}^1$, by essentially the same arguments
as we employed for the resolvent cocycle).
Removing the $1$ from the resolvents modifies the resolvent cocycle by coboundaries and cochains holomorphic at $r = -1/2$.

Before proceeding, we recall the detailed summability properties of the spectral triple 
$(\cB, \cH, \cD)$ computed in \cite{KW}.

%
%
%
%
%
%
%
%

\begin{lemma}[{\cite[Proposition 1]{KW}}]
\label{lemma_podles_basic_trace}
The function $r \mapsto \Tr(\Delta_{R}^{-1} \half (1 \pm \gamma) |\cD|^{-3-2r})$ has a 
meromorphic continuation to the complex plane which is 
holomorphic for $\Re(r) > -1/2$, 
and has a simple pole at $r = -1/2$. Furthermore, for 
all $\beta \in \cB$ we have the following equality 
\enveqn{
\mathrm{Res}_{r = -1/2} \Tr(\Delta_{R}^{-1} \half (1 \pm \gamma) \beta |\cD|^{-3-2r}) 
= \frac{(q^{-1}-q)}{2 \ln q^{-1}} \varepsilon(\beta)
}
where $\varepsilon$ is the counit of $\cA$ restricted to $\cB$ satisfying 
$\varepsilon(\re{l}{i}{0}) = \delta_{i, 0}$.
\end{lemma}

The degree zero component $\phi^r_0$ of the resolvent cocycle  is computed from the definition using the Cauchy formula and \cite{CPRS3}.
This yields the formula,  for $a_0\in\B$,

\begin{equation*}
\phi_{0}^{r}(a_{0}) = 2 \int_{0}^{\infty} \Tr \left(\Delta_{R}^{-1} \gamma \frac{1}{2\pi i} \int_{l} \lambda^{-1-r} a_{0} R_{s}(\lambda) d\lambda \right) ds= \frac{\Gamma(\frac{1}{2})\Gamma(r+\frac{1}{2})}{\Gamma(r+1)} 
\Tr \left(\Delta_{R}^{-1} \gamma a_{0} |\cD|^{-2r-1} \right).
\end{equation*}



Since $\Tr \left(\Delta_{R}^{-1} \gamma |\cD|^{-2r-1} \right) = 0$ 
for all sufficiently large $r \in \bR$, then taking the trivial continuation 
to the whole real line gives $\phi_{0}^{r}(I) = 0$ for all $r \in \bR$. This 
is the only evaluation of $\phi_{0}^{r}$ needed to compute the index 
pairing later on. 

We can compute $\mathrm{Res}_{r=-1/2}\,\phi^r_0$ explicitly, but as the calculation is quite lengthy
and we do not require this full computation for computing the index pairing, we
just quote the result; see \cite{S} for full details.

The functional $\phi_0:=\mathrm{Res}_{r=-1/2}\,\phi^r_0$ is supported on the span of the 
powers $(bc)^k$, $k=0,1,2\dots$. We have seen that $\phi_0(I)=0$. For the remaining values we have

\enveqn{
\mathrm{Res}_{r = -\frac{1}{2}} \phi_{0}^{r}(bc) = \frac{1}{2}\left( 1 - \frac{\gamma}{\ln q^{-1}}\right) - qQ
}
where $\gamma$ is Euler's constant. For $k=0,1,2,\dots$  and with $h$ the Haar state
\enveqn{
\mathrm{Res}_{r = -\frac{1}{2}} \phi_{0}^{r}((bc)^{k+2}) = \frac{(-1)^{k+1} q^{k+1}}{1 - q^{2k+2}} 
= \frac{(-1)^{k+1}}{q^{-k-1} - q^{k+1}} = -\frac{h((bc)^{k})}{q^{-1} - q}.
}

We now compute the  degree two term $\phi^r_2$ of the resolvent cocycle starting with the definition,

\begin{equation*}
\phi_{2}^{r}(a_{0}, a_{1}, a_{2}) = 4 \int_{0}^{\infty} s^{2} 
\Tr \left(\Delta_{R}^{-1} \gamma \frac{1}{2\pi i}\int_{l} 
\lambda^{-1-r} a_{0} R_{s}(\lambda) da_{1} R_{s}(\lambda) da_{2} R_{s}(\lambda) d\lambda \right) ds.
\end{equation*}

We proceed by employing the pseudodifferential calculus 
described in Lemma \ref{lemma_podles_pdc} in order to rewrite the expression 
$a_{0} R_{s}(\lambda) da_{1} R_{s}(\lambda) da_{2} R_{s}(\lambda)$ by moving 
all the resolvents to the right. From Lemma \ref{lemma_podles_pdc}, for each $\beta \in \cB$ there 
exist bounded operators $M_{1}, \, M_{2}$ such that

\begin{align*}
(\lambda - s^{2} - \cD^{2}) d\beta &= d\beta (x\lambda - s^{2}- \chi^{-2} \cD^{2}) + M_{1}|\D| ,\\
&(\lambda - s^{2} - \chi^{-2} \cD^{2})d\beta = d\beta (\lambda - s^{2} - \cD^{2}) + M_{2}|\D|.
\end{align*}

This gives the formulae
\begin{align}
R_{s}(\lambda) d\beta 
&= d\beta (\lambda - s^{2} - \chi^{-2} \cD^{2})^{-1} - R_{s}(\lambda)M_{1} |\D| (\lambda - s^{2} - \chi^{-2} \cD^{2})^{-1} \nonumber\\
d\beta R_{s}(\lambda) 
&= (\lambda - s^{2} - \chi^{-2} \cD^{2})^{-1} d\beta + (\lambda - s^{2} - \chi^{-2} \cD^{2})^{-1}M_{2} |\D| R_{s}(\lambda).
\label{eq:commute-chi}
\end{align}

Observe that the operators 
$R_{s}(\lambda)M_{1} |\D| (\lambda - s^{2} - \chi^{-2} \cD^{2})^{-1}$ and 
$(\lambda - s^{2} - \chi^{-2} \cD^{2})^{-1}M_{2} |\D| R_{s}(\lambda)$ are in 
$\wop^{-3}$ by Lemma 
\ref{lemma_podles_pdc}. Using this observation, and 
Equation \eqref{eq:commute-chi}, we can 
move all the resolvents to the right,
and in doing so we only introduce  errors which are functions holomorphic at $r=-1/2$. More precisely,
for any $a_0,\,a_1,\,a_2\in\B$, we obtain  the formula


\begin{align*}
\phi_{2}^{r}(a_{0}, a_{1}, a_{2}) &= 4 \int_{0}^{\infty} s^{2} \Tr \left(\Delta_{R}^{-1} \gamma a_{0} da_{1} da_{2} \frac{1}{2\pi i}\int_{l} \lambda^{-1-r} R_{s}(\lambda) (\lambda - s^{2} - \chi^{-2} \cD^{2})^{-1} R_{s}(\lambda) d\lambda \right) ds \\
&= 4 \int_{0}^{\infty} s^{2} \Tr \left(\Delta_{R}^{-1} \gamma a_{0} da_{1} da_{2} \frac{1}{2\pi i}\int_{l} \lambda^{-1-r} 
(\lambda - s^{2} - \chi^{-2} \cD^{2})^{-1} R_{s}(\lambda)^2 d\lambda \right) ds
\end{align*}

modulo functions holomorphic at $r = -1/2$. The integral
\enveq{
\label{eqn_cauchy_integral}
\int_{l} \lambda^{-1-r}  (\lambda - s^{2} - \chi^{-2} \cD^{2})^{-1} R_{s}(\lambda)^{2} d\lambda
}
is evaluated on the spectra of the operators $(\lambda - s^{2} - \chi^{-2} \cD^{2})^{-1}$ and $R_{s}(\lambda)^{2}$. We want to use the Cauchy integral formula, however because there are two poles to consider, $\lambda = s^{2} + \cD^{2}$ and $\lambda = s^{2} + \chi^{-2} \cD^{2}$, we outline the process.

First note that $\chi$ and $\cD^{2}$ are commuting operators with discrete spectra, and they can be simultaneously diagonalised with respect to the direct sum $\cH = \cH_{1} \oplus \cH_{-1}$. Indeed, the integrand in Equation \eqref{eqn_cauchy_integral} has the eigenbasis
\enveq{
\label{eqn_ch4_int_basis}
\left\{ \bca \re{l}{i}{1/2} \\ 0 \eca, \ \bca 0 \\ \re{l}{i}{-1/2} \eca \colon l - \half \in \bN_{0}, i \in \{ -l, -l+1, \ldots, l \} \right\},
}
on which $\chi^{-2}$ simply acts via multiplication by the scalar $q^{\pm 2} \neq 1$. We specialise to the eigenbasis in $\cH_{1}$, where $\chi^{-2}$ acts via multiplication by $q^{2}$. The argument we now present can be applied analogously to the remaining eigenbasis elements.

On each eigenvector, the integral in Equation \eqref{eqn_cauchy_integral} reduces to a scalar integral over $\lambda$, where we may apply the usual Cauchy integral formula. The integrand of this scalar integral has two poles; on the eigenbasis elements in $\cH_{1}$ described in Equation \eqref{eqn_ch4_int_basis} these poles are $\lambda_{1} = s^{2} + q^{2} [l + \half]_{q}^{2}$ and $\lambda_{2} = s^{2} + [l + \half]_{q}^{2}$, with $\lambda_{1} < \lambda_{2}$.
The contour of integration $l$ is a vertical line to the left of the spectrum for all $s\geq 0$.

 \setlength{\unitlength}{1cm}
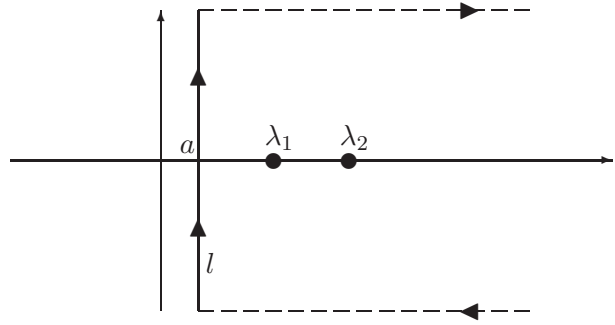
 \begin{picture}(8,5)(-5,-2.25)
 \put(0,-2){\vector(0,1){4}}
 \put(-2,0){\vector(1,0){8}}
 \put(0.5,-2){\line(0,1){4}}
 \multiput(0.5,-2)(0.3,0){15}{\line(1,0){0.2}}
 \multiput(0.5,2)(0.3,0){15}{\line(1,0){0.2}}
 \put(0.36, 1){$\blacktriangle$}
 \put(0.36, -1){$\blacktriangle$}
 \put(3.95, -2.1){$\blacktriangleleft$}
 \put(3.95, 1.9){$\blacktriangleright$}
 \put(1.5, 0){\circle*{0.2}}
 \put(2.5, 0){\circle*{0.2}}
 \put(0.25, 0.1){$a$}
 \put(0.6, -1.5){$l$}
 \put(1.38, 0.2){$\lambda_{1}$}
 \put(2.38, 0.2){$\lambda_{2}$}
 \end{picture}

In order to apply the Cauchy integral formula, we modify the contour $l$ by adding a vertical line 
$l' = \{ a' + iv \colon \lambda_{1} < a' < \lambda_{2}, \ \ v \in \bR \}$ between the poles 
$\lambda_{1}$ and $\lambda_{2}$. We integrate along this line in both directions, allowing us to split the integral into two parts.

We denote by $\Gamma_1$ the contour obtained by going up along $l$ and down along $l'$ , and denote by $\Gamma_2$ the remaining 
integration along $l'$.  Lemma \ref{lemma_Bs_trace} shows that the horizontal dashed integrals go to zero.

%


 \setlength{\unitlength}{1cm}
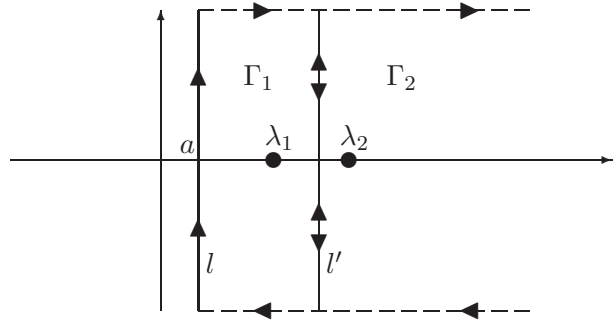
 \begin{picture}(8,5)(-5,-2.25)
 \put(0,-2){\vector(0,1){4}}
 \put(-2,0){\vector(1,0){8}}
 \put(0.5,-2){\line(0,1){4}}
 \put(2.1,-2){\line(0,1){4}}
 \multiput(0.5,-2)(0.3,0){15}{\line(1,0){0.2}}
 \multiput(0.5,2)(0.3,0){15}{\line(1,0){0.2}}
 \put(0.36, 1){$\blacktriangle$}
 \put(0.36, -1){$\blacktriangle$}
 \put(1.96, 1.2){$\blacktriangle$}
 \put(1.96, -0.8){$\blacktriangle$}
 \put(1.96, 0.8){$\blacktriangledown$}
 \put(1.96, -1.2){$\blacktriangledown$}
 \put(3.95, -2.1){$\blacktriangleleft$}
 \put(3.95, 1.9){$\blacktriangleright$}
 \put(1.2, -2.1){$\blacktriangleleft$}
 \put(1.2, 1.9){$\blacktriangleright$}
 \put(1.5, 0){\circle*{0.2}}
 \put(2.5, 0){\circle*{0.2}}
 \put(0.25, 0.1){$a$}
 \put(0.6, -1.5){$l$}
 \put(2.2, -1.5){$l'$}
 \put(1.38, 0.2){$\lambda_{1}$}
 \put(2.38, 0.2){$\lambda_{2}$}
 \put(1.1, 1){$\Gamma_{1}$}
 \put(3, 1){$\Gamma_{2}$}
 \end{picture}

Define
\enveqn{
f_{1}(\lambda) := \lambda^{-1-r} R_{s}(\lambda) R_{s}(\lambda), \qquad \qquad f_{2}(\lambda) 
:= \lambda^{-1-r} (\lambda - s^{2} - \chi^{-2} \cD^{2})^{-1}.
}
By construction, the function $f_{1}$ is holomorphic on the domain defined by the contour $\Gamma_{1}$, while $f_{2}$ is holomorphic on the domain defined by $\Gamma_{2}$. Therefore, we may apply the (scalar) Cauchy integral formula for each contour $\Gamma_{1}$ and $\Gamma_{2}$, so we write
\enveqn{
\int_{l} \lambda^{-1-r} (\lambda - s^{2} - \chi^{-2} \cD^{2})^{-1} R_{s}(\lambda)^{2} d\lambda = \int_{\Gamma_{1}} \frac{f_{1}(\lambda)}{ (\lambda - s^{2} - \chi^{-2} \cD^{2})} d\lambda + \int_{\Gamma_{2}} \frac{f_{2}(\lambda)}{(\lambda - s^{2} - \cD^{2})^{2}} d\lambda
}

This yields

\envaln{
&\frac{1}{2\pi i} \int_{l} \lambda^{-1-r} (\lambda - s^{2} - \chi^{-2} \cD^{2})^{-1} R_{s}(\lambda)^{2} d\lambda = f_{1}(s^{2}+\chi^{-2}\cD^{2}) + f_{2}'(s^{2}+\cD^{2}) \\
&\qquad\qquad= (s^{2}+\chi^{-2}\cD^{2})^{-1-r}(\chi^{-2} - 1)^{-2} \cD^{-4} -(1+r)(s^{2}+\cD^{2})^{-2-r}(1 - \chi^{-2})^{-1}\cD^{-2} \\
&\qquad\qquad\qquad\qquad\qquad\qquad\qquad\qquad\qquad-(s^{2}+\cD^{2})^{-1-r}(1 - \chi^{-2})^{-2} \cD^{-4}.
}


Inserting the result of the Cauchy integral into our previous formula for $\phi_2^r$, and
evaluating the $s$-integrals (see for example {\cite[Lemma 5.9]{CPRS3}}) yields

\envaln{
\phi_{2}^{r}(a_{0}, a_{1}, a_{2}) &= \frac{\sqrt{\pi} \Gamma\left(r - \frac{1}{2} \right)}{\Gamma(r+1)} \Tr \bigg(\Delta_{R}^{-1} \gamma a_{0} da_{1} da_{2} \chi^{2r-1}|\cD|^{-2r+1} (\chi^{-2} - 1)^{-2} \cD^{-4} \bigg) \\
& \quad - \frac{\sqrt{\pi} \Gamma\left(r + \frac{1}{2} \right)}{\Gamma(r+1)} \Tr \bigg(\Delta_{R}^{-1} \gamma a_{0} da_{1} da_{2} |\cD|^{-2r-1} (1 - \chi^{-2})^{-1}\cD^{-2} \bigg) \\
& \quad - \frac{\sqrt{\pi} \Gamma\left(r - \frac{1}{2} \right)}{\Gamma(r+1)} \Tr \bigg(\Delta_{R}^{-1} \gamma a_{0} da_{1} da_{2} |\cD|^{-2r+1} (1 - \chi^{-2})^{-2} \cD^{-4} \bigg),
}
modulo functions holomorphic at $r=-1/2$.
Writing $\Gamma(r + \half) = (r - \half) \Gamma(r-\half)$ and collecting terms, $\phi_{2}^{r}(a_{0}, a_{1}, a_{2}) $ is given by

\envaln{
& \frac{\sqrt{\pi} \Gamma\left(r - \frac{1}{2} \right)}{\Gamma(r+1)} \Tr \bigg(\Delta_{R}^{-1} \gamma a_{0} da_{1} da_{2} |\cD|^{-2r-3}  (1 - \chi^{-2})^{-2} \Big( \chi^{2r-1} - (r - \half) (1 - \chi^{-2}) - 1 \Big) \bigg) .
}

Observe that

\enval{
\Gamma\left(r - \tfrac{1}{2} \right) & \Big( \chi^{2r-1} - (r - \half) (1 - \chi^{-2}) - 1 \Big) = \Gamma\left(r - \tfrac{1}{2} \right) \Big( \chi^{-2} (\chi^{2r+1} - 1) - (r + \half) (1 - \chi^{-2}) \Big) \nonumber\\
&= (r + \half) \Gamma\left(r - \tfrac{1}{2} \right) \Big( \chi^{-2} (1 + \ln \chi^{2}) - 1 \Big) + \Gamma\left(r - \tfrac{1}{2} \right) \chi^{-2} \sum_{n = 2}^{\infty} \frac{(\ln \chi^{2})^{n}}{n!} (r + \half)^{n}.
\label{eq:gamma}
}

Now, $\Gamma\left(r - \tfrac{1}{2} \right)$ has a simple pole at $r = -1/2$, so the  function in Equation \eqref{eq:gamma}
is holomorphic at $r = -1/2$ with constant term $1 - \chi^{-2} (1 + \ln \chi^{2})$. 
Therefore

\envaln{
\phi_2(a_0,a_1,a_2):=\mathrm{Res}_{r = -1/2} \phi_{2}^{r}(a_{0}, a_{1}, a_{2}) &= \mathrm{Res}_{r = -1/2} \Tr (\Delta_{R}^{-1} \gamma a_{0} da_{1} da_{2} |\cD|^{-2r-3} C )
}

where $C = ( 1 - \chi^{-2} (1 + \ln \chi^{2}) ) (1 - \chi^{-2})^{-2} = ( \chi^{2} - 1 - \ln \chi^{2} ) (\chi - \chi^{-1})^{-2}$ is a 
constant diagonal matrix. Finally,  Equation \eqref{eqn_podles_com} yields

\enveqn{
a_{0} da_{1} da_{2} = \bma a_{0} \p_{e}(a_{1}) \p_{f}(a_{2}) & 0 \\ 0 & a_{0} \p_{f}(a_{1}) \p_{e}(a_{2}) \ema,
}

and so invoking Lemma \ref{lemma_podles_basic_trace} gives the formula

\enval{
&\phi_{2}(a_{0}, a_{1}, a_{2}) \label{eqn_phi_two} \\
&= \frac{1}{2 (q^{-1} - q) \ln q^{-1}} \Big( ( q^{-2} - 1 - \ln q^{-2}) \varepsilon(a_{0} \p_{e}(a_{1}) \p_{f}(a_{2})) - ( q^{2} - 1 - \ln q^{2}) \varepsilon(a_{0} \p_{f}(a_{1}) \p_{e}(a_{2})) \Big). \nonumber
}

\subsection{Some equivariant projections and their Chern characters}
\label{subsec:projns}

%
%
%

%

Our aim is to 
construct representatives in the equivariant $K$-theory $K_{0}^{\bT}(\cB)$ for the action of the modular automorphism group 
$\Psi_{R}$, which is given by $\sigma^{\Psi_{R}}_{t}=\vartheta_{t}^{-1}$.
These representatives will take the form of projections $p \in M_{N \times N}(\cB)$ together with a 
representation $V \colon \bT \rightarrow M_{N \times N}(\bC)$ such that $p$ is $\vartheta^{-1} \ox \mathrm{Ad}(V)$-invariant. See \cite{W} for 
similar constructions.

For $n \in \frac{1}{2} \bZ$, define

\enveqn{
T_{n}^{l} := \left( \begin{array}{c} \re{l}{l}{n} \\ \re{l}{l-1}{n} \\ \vdots \\ \re{l}{-l}{n} \end{array} \right), \qquad \text{and} \qquad P_{n} := T_{n}^{|n|} T_{n}^{|n|\ast}.
}

More explicitly,

\enveqn{
P_{n} = \left( \begin{array}{cccc}
\re{|n|}{|n|}{n}\re{|n|\ast}{|n|}{n} & \re{|n|}{|n|}{n}\re{|n|\ast}{|n|-1}{n} & \cdots & \re{|n|}{|n|}{n}\re{|n|\ast}{-|n|}{n} \\
\re{|n|}{|n|-1}{n}\re{|n|\ast}{|n|}{n} & \ddots & \ddots & \vdots \\
\vdots & \ddots & \ddots & \vdots \\
\re{|n|}{-|n|}{n}\re{|n|\ast}{|n|}{n} & \cdots & \cdots & \re{|n|}{-|n|}{n}\re{|n|\ast}{-|n|}{n}
\end{array} \right),
\qquad (P_{n})_{r, s} = \re{|n|}{|n|-r+1}{n}\re{|n|\ast}{|n|-s+1}{n}.
}

By construction, $P_{n}^{\ast} = P_{n}$ and $P_{n} \in M_{(2|n|+1) \times (2|n|+1)}(\cB)$. Furthermore,

\enveqn{
T_{n}^{l\ast}T_{n}^{l} = \sum_{p = -l}^{l} \re{l\ast}{p}{n} \re{l}{p}{n} ={\rm  Id}_n
}

and hence $P_{n}^{2} = P_{n}$, so $P_{n}$ is a projection. Now define $\lambda_{j} = q^{-2j+2} \in [1, \infty)$ for $j \in \{ 1, 2, \ldots, 2|n|+1 \}$, and 
define $V_n:\T\to M_{2|n|+1}(\C)$ by

\enveqn{
V_{n}(t) := \left( \begin{array}{cccc}
\lambda_{1}^{it} & 0 & \cdots & 0 \\
0 & \lambda_{2}^{it} & \ddots & \vdots \\
\vdots & \ddots & \ddots & 0 \\
0 & \cdots & 0 & \lambda_{2|n|+1}^{it}
\end{array} \right).
}

While we have defined $V_n$ to be a real action on $M_{(2|n|+1) \times (2|n|+1)}(\bC)$, the action is periodic and so induces a circle action. Observe that

\enveqn{
\sigma_{t}((P_{n})_{r, s}) = q^{-2it(|n|-r+1 - (|n|-s+1))} \re{|n|}{|n|-r+1}{n}\re{|n|\ast}{|n|-s+1}{n} = q^{2it(r-s)} (P_{n})_{r, s}
}

and

\enveqn{
\mathrm{Ad}(V_{n}(t))(P_{n})_{r, s} = (V_{n}(t) P_{n} V_{n}^{-1}(t))_{r, s} = \lambda_{r}^{it} (P_{n})_{r, s} (\lambda_{s}^{it})^{-1} = q^{2it(s-r)} (P_{n})_{r, s}.
}

So $P_{n}$ is $\vartheta^{-1}\ox \mathrm{Ad}(V_{n})$-invariant and we  define the weight $G \colon M_{(2|n|+1) \times (2|n|+1)}(\bC) \rightarrow \bC$ by

\enveqn{
G(X) := \Tr(V_{n,-i} X)
}

for $X \in M_{(2|n|+1) \times (2|n|+1)}(\bC)$, and $(V_{n,-i})_{k, m} = \delta_{k, m} q^{-2k+2}$.


We have demonstrated that $P_{n}$ is an equivariant projection for the 
circle action represented by $V_n$, and therefore defines a class in 
$K_{0}^{\bT}(\cB)$. We now write down the Chern character of this 
representative, Equation \eqref{eq:even-chern}.

\begin{lemma}
\label{lemma_chern_pn}

The Chern character of $[P_{n},V_n]$  is

\envaln{
Ch_{0}([P_{n}, V_n]) &= q^{2(n - |n|)} I, \\
Ch_{2}([P_{n}, V_n]) &= \!-2 \sum_{k_{0}, k_{1}, k_{2}=0}^{2|n|} \!\!\!q^{-2k_{0}}\left( \re{|n|}{|n|-k_{0}}{n}\re{|n|\ast}{|n|-k_{1}}{n}-\frac{1}{2}\delta_{k_0,k_1}\right) \ox \re{|n|}{|n|-k_{1}}{n}\re{|n|\ast}{|n|-k_{2}}{n} \ox \re{|n|}{|n|-k_{2}}{n}\re{|n|\ast}{|n|-k_{0}}{n}.
}

\end{lemma}
\begin{proof}
Using Equation \eqref{eq:even-chern}
we have
\enveqn{
Ch_{0}([P_{n}, V_n]) = \sum_{k_{0}, k_{1}=1}^{2|n|+1} (V_{n,-i})_{k_{1}, k_{0}} (P_{n})_{k_{0}, k_{1}} = \sum_{k_{0}, k_{1}=1}^{2|n|+1} \delta_{k_{1}, k_{0}} q^{-2k_{0}+2} \re{|n|}{|n|-k_{0}+1}{n}\re{|n|\ast}{|n|-k_{1}+1}{n}.
}

Now we apply the formulae for adjoints $(\re{l}{i}{j})^* = (-q)^{j-i} \re{l}{-i}{-j}$ and $\re{l}{i}{j} = (-q)^{j-i} (\re{l}{-i}{-j})^*$, along with the 
unitary relations for the Peter-Weyl basis elements, 
\cite[Proposition 16, Chapter 4]{KS},
 to obtain

\envaln{
Ch_{0}([P_{n}, V_n]) &= \sum_{k=0}^{2|n|} q^{-2k} q^{2(n - |n| + k)} \re{|n|\ast}{k-|n|}{-n}\re{|n|}{k-|n|}{-n} = q^{2(n - |n|)} I.
}

Finally,

\begin{align*}
Ch_{2}([P_{n}, V_n]) &= -\frac{2!}{1!} \sum_{k_{0}, k_{1}, k_{2}, k_{3}=1}^{2|n|+1} (V_{n,-i})_{k_{3}, k_{0}} (P_{n} - \tfrac{1}{2})_{k_{0}, k_{1}} \ox (P_{n})_{k_{1}, k_{2}} \ox (P_{n})_{k_{2}, k_{3}} \\
&= -2 \sum_{k_{0}, k_{1}, k_{2}=1}^{2|n|+1} (V_{n,-i})_{k_{0}, k_{0}} (P_{n} - \tfrac{1}{2})_{k_{0}, k_{1}} \ox (P_{n})_{k_{1}, k_{2}} \ox (P_{n})_{k_{2}, k_{0}} \\
&= -2 \!\!\!\!\!\!\sum_{k_{0}, k_{1}, k_{2}=0}^{2|n|} 
\!\!\!\!\!q^{-2k_{0}}\!\! \left( \re{|n|}{|n|-k_{0}}{n}\re{|n|\ast}{|n|-k_{1}}{n} - \tfrac{1}{2} \delta_{k_{0}, k_{1}} \right)\! \ox \re{|n|}{|n|-k_{1}}{n}\re{|n|\ast}{|n|-k_{2}}{n} 
\!\ox \re{|n|}{|n|-k_{2}}{n}\re{|n|\ast}{|n|-k_{0}}{n}. \qed
\end{align*}
\hideqed\end{proof}

\subsection{The index pairing}

The resolvent index formula established in Section \ref{subsec:local-index} proves that the index pairing  defined by the 
modular spectral triple $(\cB, \cH, \cD, \B(\cH), \Psi_{R})$ and the equivariant $K$-theory class defined by the projection $P_{n}$ is given by the formula

\enveqn{
\mathrm{Ind}_{\Psi_{R} \ox G}(P_{n} (\cD \ox \mathrm{Id}_{2|n|+1})^{+} P_{n}) =   \phi_{2}(Ch_{2}([P_{n}, V_n])) + \phi_{0}(Ch_{0}([P_{n}, V_n])) .
}

Now that we have explicit formulae for the cocycle $(\phi_0,\phi_2)$ and the cycle $Ch_{\ast}([P_{n}, V_n])$, the computation is straightforward.

\begin{prop}
\label{prop_phi_p_two}

The evaluation of $\phi_{2}$ on $Ch_{2}([P_{n}, V_n])$ is

\enveqn{
\phi_{2}(Ch_{2}([P_{n}, V_n])) = q^{-2|n|} [2n]_{q}.
}

\end{prop}
\begin{proof}

Recalling the formula for $\phi_{2}$ from 
Equation \eqref{eqn_phi_two} and the expression for $Ch_{2}([P_{n}, V_n])$ from 
Lemma \ref{lemma_chern_pn}, we compute

\envaln{
& \varepsilon\left(\left( \re{|n|}{|n|-k_{0}}{n}\re{|n|\ast}{|n|-k_{1}}{n} - \tfrac{1}{2} \delta_{k_{0}, k_{1}} \right) \p_{e}(\re{|n|}{|n|-k_{1}}{n}\re{|n|\ast}{|n|-k_{2}}{n}) \p_{f}(\re{|n|}{|n|-k_{2}}{n}\re{|n|\ast}{|n|-k_{0}}{n}) \right) \\
& \qquad = \left( \delta_{|n|-k_{0}, n} \delta_{|n|-k_{1}, n} - \tfrac{1}{2} \delta_{k_{0}, k_{1}} \right) \varepsilon(\p_{e}(\re{|n|}{|n|-k_{1}}{n}\re{|n|\ast}{|n|-k_{2}}{n})) \varepsilon(\p_{f}(\re{|n|}{|n|-k_{2}}{n}\re{|n|\ast}{|n|-k_{0}}{n})) \\
& \qquad = \delta_{k_{0}, k_{1}} \left( \delta_{k_{0}, |n|-n} - \tfrac{1}{2} \right) \varepsilon(\p_{e}(\re{|n|}{|n|-k_{1}}{n}\re{|n|\ast}{|n|-k_{2}}{n})) \varepsilon(\p_{f}(\re{|n|}{|n|-k_{2}}{n}\re{|n|\ast}{|n|-k_{1}}{n})).
}

We observe that this expression is zero for the case $n = 0$, because $\p_{e}(I) = \p_{f}(I) = 0$. So for the remainder we consider only nonzero $n$. Observe that $\re{|n|}{|n|-k_{2}}{n}\re{|n|\ast}{|n|-k_{1}}{n} = (\re{|n|}{|n|-k_{1}}{n}\re{|n|\ast}{|n|-k_{2}}{n})^{\ast}$. Now we use the property that $(g \triangleright \alpha)^{\ast} = S(g)^{\ast} \triangleright \alpha^{\ast}$ for all $g \in \cU_{q}(su_{2})$ and $\alpha \in \cA$, so that $\varepsilon(\p_{e}(\alpha^{\ast})) = -q \varepsilon(\p_{f}(\alpha))$ and $\varepsilon(\p_{f}(\alpha^{\ast})) = -q^{-1} \varepsilon(\p_{e}(\alpha))$. Then

\envaln{
& \varepsilon\left(\left( \re{|n|}{|n|-k_{0}}{n}\re{|n|\ast}{|n|-k_{1}}{n} - \tfrac{1}{2} \delta_{k_{0}, k_{1}} \right) \p_{e}(\re{|n|}{|n|-k_{1}}{n}\re{|n|\ast}{|n|-k_{2}}{n}) \p_{f}(\re{|n|}{|n|-k_{2}}{n}\re{|n|\ast}{|n|-k_{0}}{n}) \right) \\
& \qquad = -q^{-1} \delta_{k_{0}, k_{1}} \left( \delta_{k_{0}, |n|-n} - \tfrac{1}{2} \right) \varepsilon(\p_{e}(\re{|n|}{|n|-k_{1}}{n}\re{|n|\ast}{|n|-k_{2}}{n}))^{2},
}

and similarly

\envaln{
& \varepsilon\left(\left( \re{|n|}{|n|-k_{0}}{n}\re{|n|\ast}{|n|-k_{1}}{n} - \tfrac{1}{2} \delta_{k_{0}, k_{1}} \right) \p_{f}(\re{|n|}{|n|-k_{1}}{n}\re{|n|\ast}{|n|-k_{2}}{n}) \p_{e}(\re{|n|}{|n|-k_{2}}{n}\re{|n|\ast}{|n|-k_{0}}{n}) \right) \\
& \qquad = -q^{-1} \delta_{k_{0}, k_{1}} \left( \delta_{k_{0}, |n|-n} - \tfrac{1}{2} \right) \varepsilon(\p_{e}(\re{|n|}{|n|-k_{2}}{n}\re{|n|\ast}{|n|-k_{1}}{n}))^{2}.
}

Using the twisted derivation property of $\p_{e}$ on $\cA$, we find, for $r, s \in \{ 0, \ldots, 2|n|\}$,

\envaln{
\varepsilon(\p_{e}(\re{|n|}{|n|-r}{n}\re{|n|\ast}{|n|-s}{n}))^{2} &= \varepsilon(\p_{e}(\re{|n|}{|n|-r}{n}) \p_{k}(\re{|n|\ast}{|n|-s}{n}) + \p_{k}^{-1}(\re{|n|}{|n|-r}{n}) \p_{e}(\re{|n|\ast}{|n|-s}{n}))^{2} \\
&= q^{-2n} \left( \varepsilon(\p_{e}(\re{|n|}{|n|-r}{n})) \delta_{s, |n|-n} + \delta_{r, |n|-n} \varepsilon(\p_{e}(\re{|n|\ast}{|n|-s}{n})) \right)^{2} \\
&= q^{-2n} \left( \varepsilon(\p_{e}(\re{|n|}{|n|-r}{n})) \delta_{s, |n|-n} - q\delta_{r, |n|-n} \varepsilon(\p_{f}(\re{|n|}{|n|-s}{n})) \right)^{2} \\
&= q^{-2n} \left( \varepsilon(\kappa^{|n|}_{n+1} \re{|n|}{|n|-r}{n+1}) \delta_{s, |n|-n} - q\delta_{r, |n|-n} \varepsilon(\kappa^{|n|}_{n} \re{|n|}{|n|-s}{n-1}) \right)^{2} \\
&= q^{-2n} \left( \kappa^{|n|}_{n+1} \delta_{r, |n|-n-1} \delta_{s, |n|-n} - q\kappa^{|n|}_{n} \delta_{r, |n|-n} \delta_{s, |n|-n+1} \right)^{2} \\
&= q^{-2n} \left( (\kappa^{|n|}_{n+1})^{2} \delta_{r, |n|-n-1} \delta_{s, |n|-n} + q^{2} (\kappa^{|n|}_{n})^{2} \delta_{r, |n|-n} \delta_{s, |n|-n+1} \right)
}

where $\kappa^{l}_{j} = ([l+j]_{q}[l-j+1]_{q})^{1/2}$. Combining these results with the formula for $\mathrm{Res}_{r = -\frac{1}{2}}\phi_{2}^{r}$ and the expression for $Ch_{2}([P_{n}, V])$ gives

\envaln{
&\phi_{2}(Ch_{2}([P_{n}, V_n])) = \frac{-2 }{2 (q^{-1} - q) \ln q^{-1}} \sum_{k_{0}, k_{1}, k_{2}=0}^{2|n|} q^{-2k_{0}} (-q^{-1} \delta_{k_{0}, k_{1}} ( \delta_{k_{0}, |n|-n} - \tfrac{1}{2})) \\
& \qquad \times \Big( ( q^{-2} - 1 - \ln q^{-2}) \varepsilon(\p_{e}(\re{|n|}{|n|-k_{1}}{n}\re{|n|\ast}{|n|-k_{2}}{n}))^{2}
 - ( q^{2} - 1 - \ln q^{2}) \varepsilon(\p_{e}(\re{|n|}{|n|-k_{2}}{n}\re{|n|\ast}{|n|-k_{1}}{n}))^{2} \Big) \\
&= \frac{q^{-1} }{(q^{-1} - q) \ln q^{-1}} \sum_{k_{1}, k_{2}=0}^{2|n|} q^{-2k_{1}} ( \delta_{k_{1}, |n|-n} - \tfrac{1}{2}) q^{-2n} \\
& \qquad \times \Big( ( q^{-2} - 1 - \ln q^{-2}) \left( (\kappa^{|n|}_{n+1})^{2} \delta_{k_{1}, |n|-n-1} \delta_{k_{2}, |n|-n} + q^{2} (\kappa^{|n|}_{n})^{2} \delta_{k_{1}, |n|-n} \delta_{k_{2}, |n|-n+1} \right) \\
& \qquad \qquad - ( q^{2} - 1 - \ln q^{2}) \left( (\kappa^{|n|}_{n+1})^{2} \delta_{k_{2}, |n|-n-1} \delta_{k_{1}, |n|-n} + q^{2} (\kappa^{|n|}_{n})^{2} \delta_{k_{2}, |n|-n} \delta_{k_{1}, |n|-n+1} \right) \Big).
}

Using $( \delta_{k_{1}, |n|-n} - \tfrac{1}{2}) \delta_{k_{1}, |n|-n} = \tfrac{1}{2} \delta_{k_{1}, |n|-n}$ 
and $( \delta_{k_{1}, |n|-n} - \tfrac{1}{2}) \delta_{k_{1}, |n|-n \pm 1} = -\tfrac{1}{2} \delta_{k_{1}, |n|-n \pm 1}$ yields

\envaln{
&\mathrm{Res}_{r = -\frac{1}{2}}\phi_{2}^{r}(Ch_{2}([P_{n}, V_n])) = \frac{q^{-1} }{(q^{-1} - q) \ln q^{-1}} \sum_{k_{1}, k_{2}=0}^{2|n|} \tfrac{1}{2} q^{-2k_{1}-2n} \\
& \qquad \times \Big( ( q^{-2} - 1 - \ln q^{-2}) \left( -(\kappa^{|n|}_{n+1})^{2} \delta_{k_{1}, |n|-n-1} \delta_{k_{2}, |n|-n} + q^{2} (\kappa^{|n|}_{n})^{2} \delta_{k_{1}, |n|-n} \delta_{k_{2}, |n|-n+1} \right) \\
& \qquad \qquad - ( q^{2} - 1 - \ln q^{2}) \left( (\kappa^{|n|}_{n+1})^{2} \delta_{k_{1}, |n|-n} \delta_{k_{2}, |n|-n-1} - q^{2} (\kappa^{|n|}_{n})^{2} \delta_{k_{1}, |n|-n+1} \delta_{k_{2}, |n|-n} \right) \Big).
}

We can reduce the different summations over $k_{1}$ and $k_{2}$ down to two distinct sums, either
\enveqn{
\sum_{k = 0}^{2|n|} \delta_{k, |n|-n - 1} = \delta_{n, -|n|},  \qquad \qquad \text{or} \qquad \qquad
\sum_{k = 0}^{2|n|} \delta_{k, |n|-n + 1} = \delta_{n, |n|}.
}
Hence
\envaln{
&\phi_{2}(Ch_{2}([P_{n}, V])) = \frac{q^{-1} }{2(q^{-1} - q) \ln q^{-1}} \\
& \qquad \times \Big( ( q^{-2} - 1 - \ln q^{-2}) \left( -(\kappa^{|n|}_{n+1})^{2} \delta_{n, -|n|} q^{-2(|n|-n-1)-2n} + q^{2} (\kappa^{|n|}_{n})^{2} \delta_{n, |n|} q^{-2(|n|-n)-2n} \right) \\
& \qquad \qquad - ( q^{2} - 1 - \ln q^{2}) \left( (\kappa^{|n|}_{n+1})^{2} \delta_{n, -|n|} q^{-2(|n|-n)-2n} - q^{2} (\kappa^{|n|}_{n})^{2} \delta_{n, |n|} q^{-2(|n|-n+1)-2n} \right) \Big) \\
&= \frac{q^{-1} }{2(q^{-1} - q) \ln q^{-1}} \\
& \qquad \times \Big( ( q^{-2} - 1 - \ln q^{-2}) \left( -(\kappa^{|n|}_{1-|n|})^{2} \delta_{n, -|n|} q^{-2|n|+2} + q^{2} (\kappa^{|n|}_{|n|})^{2} \delta_{n, |n|} q^{-2|n|} \right) \\
& \qquad \qquad - ( q^{2} - 1 - \ln q^{2}) \left( (\kappa^{|n|}_{1-|n|})^{2} \delta_{n, -|n|} q^{-2|n|} - q^{2} (\kappa^{|n|}_{|n|})^{2} \delta_{n, |n|} q^{-2|n|-2} \right) \Big).
}

Observe that $(\kappa^{|n|}_{1-|n|})^{2} = (\kappa^{|n|}_{|n|})^{2} = [2|n|]_{q}$ as $[1]_{q} = 1$, and so

\envaln{
&\phi_{2}(Ch_{2}([P_{n}, V_n])) = \frac{q^{-1} }{2(q^{-1} - q) \ln q^{-1}} [2|n|]_{q} q^{-2|n|} \\
& \qquad \times \Big( ( q^{-2} - 1 - \ln q^{-2}) q^{2} (\delta_{n, |n|} - \delta_{n, -|n|}) - ( q^{2} - 1 - \ln q^{2}) (\delta_{n, -|n|} - \delta_{n, |n|}) \Big) \\
&= \frac{q^{-2|n|-1} [2|n|]_{q} }{2(q^{-1} - q) \ln q^{-1}} (\delta_{n, |n|} - \delta_{n, -|n|}) \Big( ( q^{-2} - 1 - \ln q^{-2}) q^{2} + ( q^{2} - 1 - \ln q^{2}) \Big) \\
&= \frac{q^{-2|n|-1} [2|n|]_{q} }{2(q^{-1} - q) \ln q^{-1}} (\delta_{n, |n|} - \delta_{n, -|n|}) \Big( - q^{2}\ln q^{-2} - \ln q^{2} \Big) \\
&= \frac{q^{-2|n|-1} [2|n|]_{q} }{2(q^{-1} - q) \ln q^{-1}} (\delta_{n, |n|} - \delta_{n, -|n|}) (1 - q^{2}) \ln q^{-2}  \\
&= q^{-2|n|} [2|n|]_{q} (\delta_{n, |n|} - \delta_{n, -|n|}).
}

Considering $n \neq 0$, then $(\delta_{n, |n|} - \delta_{n, -|n|}) = \mathrm{sgn}(n)$ and 
$\mathrm{sgn}(n)[2|n|]_{q} = [2n]_{q}$. As $[0]_{q} = 0$, then for all $n \in \frac{1}{2} \bZ$ we have

\enveqn{
\phi_{2}(Ch_{2}([P_{n}, V_n])) = q^{-2|n|} [2n]_{q}. \qedhere
}
\end{proof}

We can now write down the index pairing and compute the classical limit as $q \rightarrow 1$.

\begin{thm}
For $N \in \bZ$, the index pairing of the modular spectral triple $(\cB, \cH, \cD, \B(\cH), \Psi_{R})$ with the equivariant projections $P_{N/2}$ is
\enveqn{
\mathrm{Ind}(P_{N/2} (\cD\ox{\rm Id}_{|N| + 1})^{+} P_{N/2}) = q^{-|N|} [N]_{q}.
}
The classical limit of the index as $q \rightarrow 1$ is
\enveqn{
\lim_{q \rightarrow 1} \mathrm{Ind}(P_{N/2} (\cD\ox{\rm Id}_{|N| + 1})^{+} P_{N/2}) = N.
}
\end{thm}
\begin{proof}
First, the degree zero contribution is  $\phi_{0}(Ch_{0}([P_{N/2}, V_{N/2}])) = 0$.
This follows from $\phi_{0}(I) = 0$, and from Lemma \ref{lemma_chern_pn}, which gives
$Ch_{0}([P_{N/2}, V_{N/2}]) = q^{(N - |N|)} I$. 
Thus the index pairing is computed just with the degree 2 part, which comes from Proposition \ref{prop_phi_p_two}. To compute the classical limit of the 
index we recall that $\lim_{q \rightarrow 1} [N]_{q} = N$ (see for example \cite{KS}).
\end{proof}

\end{document}